\documentclass{amsart}
\usepackage{mathabx,amsfonts,amssymb,amsrefs}

\newcommand\Z{\mathbb{Z}}
\newcommand\N{\mathbb{N}}
\newcommand\R{\mathbb{R}}
\newcommand\sym{{\mathfrak S}}
\newcommand\pair[1]{{\langle\!\langle}#1{\rangle\!\rangle}}
\newcommand\zot{{\mathbf{012}}}

\newcommand\coefficient{coefficient}
\newcommand\coordinate{coordinate}
\newcommand\arity{degree}

\newtheorem{lemma}{Lemma}[section]

\newtheorem{proposition}[lemma]{Proposition}

\newtheorem{corollary}[lemma]{Corollary}

\newtheorem{maintheorem}{Theorem}

\theoremstyle{definition}
\newtheorem{remark}[lemma]{Remark}
\newtheorem{definition}[lemma]{Definition}

\begin{document}
\title{Groups of given intermediate word growth}
\author{Laurent Bartholdi}
\address{L.B.: Mathematisches Institut, Georg-August Universit\"at, G\"ottingen, Germany}

\author{Anna Erschler}
\address{A.E.: C.N.R.S., D\'epartement de Math\'ematiques, Universit\'e Paris Sud, Orsay, France}

\date{12 April 2012}

\thanks{The work is supported by the ERC starting grant GA 257110
  ``RaWG'', the ANR ``DiscGroup'' and the Courant Research Centre
  ``Higher Order Structures'' of the University of G\"ottingen}

\begin{abstract}
  We show that there exists a finitely generated group of growth $\sim
  f$ for all functions $f\colon\R_+\to\R_+$ satisfying $f(2R)\le
  f(R)^2\le f(\eta_+R)$ for all $R$ large enough and
  $\eta_+\approx2.4675$ the positive root of $X^3-X^2-2X-4$. Set
  $\alpha_-=\log2/\log\eta_+\approx0.7674$; then all functions that
  grow uniformly faster than $\exp(R^{\alpha_-})$ are realizable as
  the growth of a group.

  We also give a family of sum-contracting branched groups of growth
  $\sim\exp(R^\alpha)$ for a dense set of $\alpha\in[\alpha_-,1].$
\end{abstract}
\maketitle

\section{Introduction}
In~\cite{grigorchuk:growth}, Grigorchuk discovered the first example
of group with growth function strictly between polynomial and
exponential. Recall that the growth function of a finitely generated
group $G$, generated by $S$ as a semigroup, is
\[v(R)=\#\{g\in G\mid g=s_1\cdots s_\ell,\,\ell\le R,s_i\in S\};\] it
measures the volume of a ball of radius $R$ in the Cayley graph of
$G$. This function depends on the choice of generating set $S$, but
only mildly: say $v\sim v'$ if there is a constant $C>0$ such that
$v(R)\le v'(CR)$ and $v'(R)\le v(CR)$ for all $R$ large enough. Then
the $\sim$-equivalence class of $v$ is independent of the choice of
$S$.

Grigorchuk's result has been considerably extended, mainly by
Grigorchuk himself, who constructed uncountably many groups of
intermediate growth~\cite{grigorchuk:gdegree}.
See~\cites{bartholdi:lowerbd,bartholdi:upperbd,bartholdi-s:wpg,erschler:degrees,leonov:lowerbd,muchnik-p:growth},
the books~\cites{harpe:ggt,mann:howgroupsgrow}, or~\S\ref{ss:prior}
for a brief overview.

Nevertheless, these results only give estimates on the growth of these
groups. The first actual computation (up to $\sim$) of the growth
function of a group of intermediate growth appears
in~\cite{bartholdi-erschler:permutational}. In the present paper, we
prove that a large class of intermediate growth functions are the
growth functions of finitely generated groups. Our main result is:
\begin{maintheorem}\label{thm:main}
  Let $\eta_+\cong2.4675$ be the positive root of $X^3-X^2-2X-4$.  Let
  $f\colon\R_+\to\R_+$ be a function satisfying
  \begin{equation}\label{eq:gdoubling}
    f(2R)\le f(R)^2 \le f(\eta_+R)\text{ for all $R$ large enough}.
  \end{equation}
  Then there exists a finitely generated group with growth $\sim f$.
\end{maintheorem}

Set $\alpha_-=\log2/\log\eta_+\approx0.7674$; then every
submultiplicative function that grows uniformly faster (in the sense
of~\eqref{eq:gdoubling}) than $\exp(R^{\alpha_-})$ is equivalent to
the growth function of a group.

The growth function of a group is necessarily nondecreasing and
submultiplicative ($f(R+R')\le f(R)f(R')$). Note
that~\eqref{eq:gdoubling} implies that $f$ is equivalent to a
monotone, submultiplicative function.

The following examples of functions all satisfy~\eqref{eq:gdoubling};
so special cases of Theorem~\ref{thm:main} give:
\begin{itemize}
\item for every $\alpha\in[\alpha_-,1]$, there exists a group of
  growth $\sim\exp(R^\alpha)$;
\item there exists groups of growth $\sim\exp(R/\log R)$, of growth
  $\sim\exp(R/\log\log R)$, of growth $\exp(R/\log\cdots\log R)$;
\item there exists a group of growth $\sim\exp(R/A(R,R)^{-1})$, for
  $A(m,n)$ the Ackermann function; this last growth function is faster
  than any subexponential primitive-recursive function;
\item for every $\alpha\le\beta\in[\alpha_-,1)$, there exists a group
  whose growth accumulates both at $\exp(R^\alpha)$ and at
  $\exp(R^\beta)$; this recovers a result by Brieussel,
  see~\cite{brieussel:growth} and~\S\ref{ss:prior}.
\end{itemize}

\begin{remark}\label{rem:main}
  Furthermore, the group in Theorem~\ref{thm:main} may be chosen to be
  of the form $W_\omega:=A\wr_X G_\omega$ for a Grigorchuk group
  $G_\omega$ and a finite group $A$; so it is residually finite as
  soon as $A$ is abelian. Briefly, $G_\omega$ is defined by its action
  on the set of infinite sequences $\{1,2\}^\infty$, and $X$ is the
  orbit of $2^\infty$; the wreath product $W_\omega$ is the semidirect
  product $(\bigoplus_X A)\rtimes G_\omega$ for the natural action of
  $G_\omega$ on $X$. See~\S\ref{ss:ss} for details.

  We give illustrations in~\S\ref{ss:examples} of sequences $\omega$
  and the growth of the corresponding $W_\omega$.


  If additionally $f(R)^{2+2\log\eta_+/\log R}\le f(\eta_+R)$ for all
  $R$ large enough, then there exists a torsion-free,
  residually finite, finitely generated group with growth $\sim
  f$. This group may be chosen to be an extension $\tilde G_\omega$ of
  $G_\omega$ with abelian kernel. See~\S\ref{ss:torsionfree} for details.

  If $f$ is recursive, both of these groups may be chosen to be
  recursively presented, and to have solvable word problem;
  see~\S\ref{ss:presentations} and Remark~\ref{rem:solvablewp} for
  details.
\end{remark}

We obtain more precise information on $W_\omega$ in some particular
cases. A group $G$ is \emph{self-similar} if it is endowed with a
homomorphism $\phi:G\to G\wr_{\{1,\dots,d\}}\sym_d$ from $G$ to its
permutational wreath product with the symmetric group $\sym_d$. The
self-similar group $G$ is \emph{branched} if there exists a
finite-index subgroup $K\le G$ such that $\phi(K)$ contains $K^d$;
see~\S\ref{ss:ss} for more details, or~\cite{bartholdi-g-s:bg} for a
survey of consequences of the property. A self-similar group $G$ with
proper metric $\|\cdot\|$ is \emph{sum-contracting} if $\phi$ is
injective and there exist constants $\lambda<1$ (the \emph{contraction
  \coefficient}) and $C$ such that, writing
$\phi(g)=\pair{g_1,\dots,g_d}\pi$, we have
\[\sum_{i=1}^d\|g_i\|\le\lambda\|g\|+C.\]

A weaker contraction property `$\max\|g_i\|\le\lambda\|g\|+C$' is used
extensively to study self-similar groups, and is at the heart of all
inductive proofs. A key observation by Grigorchuk in the early 1980's
was that the (somewhat counterintuitive) sum-contracting property
holds for a finitely generated group (he showed that the group
$G_{\overline\zot}$ is sum-contracting with $d=8$ and
$\lambda=\frac34$). Let $G$ be sum-contracting with contraction
\coefficient\ $\lambda$, and write $\eta=d/\lambda$. Then $G$ has
growth at most $\exp(R^{\log d/\log\eta})$. Note that there are
countably many sum-contracting groups, since such a group is
determined by a choice of generators, of map $\phi$, and of which of
the words of length $\le C/(1-\lambda)$ are trivial;
see~\S\ref{ss:presentations} for details.

\begin{maintheorem}\label{thm:periodic}
  For every periodic sequence $\omega$ containing all three letters
  $\mathbf0,\mathbf1,\mathbf2$ and every non-trivial finite group $A$,
  the finitely generated group $W_\omega:=A\wr_X G_\omega$ is
  branched, sum-contracting, and has growth $\sim\exp(R^\alpha)$ for
  some $\alpha\in(\alpha_-,1)$.

  Furthermore, the set of growth exponents $\alpha$ arising in this
  manner is dense in $[\alpha_-,1]$.

  For these same $\alpha$, there exists a torsion-free self-similar,
  residually finite group $\tilde G_\omega$ of growth $\sim
  \exp(R^\alpha\log R)$.
\end{maintheorem}

For example, $W_{\overline\zot}$ has growth
$\sim\exp(R^{\alpha_-})$, as was shown
in~\cite{bartholdi-erschler:permutational}. The growth of $W_\omega$
for some other periodic sequences is as follows:
\begin{itemize}
\item For $\omega=\overline{\mathbf 0^2\mathbf{12}}$, the growth of $W_\omega$
  is $\sim\exp(R^{\log2/\log\eta})$ with $\eta\sim2.4057$ the positive
  root of $X^{12}-39X^8+192X^4-256$.
\item For $\omega=\overline{\mathbf{0102}}$, the growth of $W_\omega$
  is $\sim\exp(R^{\log2/\log\eta})$ with $\eta\sim2.4283$ the positive
  real root of $X^6+5X^4-8X^2-16$.
\item For $\omega=\overline{\mathbf 0^{t-2}\mathbf{12}}$ and
  $t\to\infty$, the constant $\eta$ converges to $2$; more precisely,
  $\eta$ is the positive real root of
  $X^{3t}-(2^{t+1}+2^{t-1}-1)X^{2t}+(2^{2t-1}+2^{t+2})X^t-2^{2t}$.

  For $t$ large, that polynomial is approximately proportional to
  $(X/2)^{2t}-\frac52(X/2)^t+\frac12$, so
  $\eta\approx2\sqrt[t]{(5+\surd17)/4}\approx2(1+C/t)$ with
  $C=\log((5+\surd17)/4)$. This gives growth
  $\sim\exp(R^{1-C'/(t\log2)})$ for $C'\to C$ as $t\to\infty$.
\item For $\omega=\overline{\mathbf 0^t\mathbf 1^t\mathbf 2^t}$ and
  $t\to\infty$, the constant $\eta$ also converges to $2$; more
  precisely, $\eta$ is the positive real root of
  \[X^{9t}-(4\cdot2^{3t}-6\cdot2^{2t}+6\cdot2^t-1)X^{6t}-(2^{6t}-6\cdot2^{5t}+6\cdot2^{4t}-4\cdot2^{3t})X^{3t}-2^{6t}.\]

  For $t$ large, that polynomial is approximately a multiple of
  $(X/2)^{6t}-4(X/2)^{3t}-1$, so
  $\eta\approx2\sqrt[3t]{2+\surd5}\approx2(1+C/t)$ with
  $C=\log(2+\surd5)$. This gives growth
  $\sim\exp(R^{1-C'/(t\log2)})$ for $C'\to C$ as $t\to\infty$.
\item If two finite sequences $\omega,\omega'$ may be obtained one
  from the other by cyclic permutation of the letters and permutation
  of the labels $\mathbf0,\mathbf1,\mathbf2$, then the groups
  $W_{\overline\omega}$ and $W_{\overline{\omega'}}$ have
  commensurable direct powers, so their growth functions are
  equivalent. Furthermore, if $\omega$ is the reverse of $\omega'$,
  then $W_{\overline\omega}$ and $W_{\overline{\omega'}}$ have same
  growth; this also happens in other cases,
  e.g. $\omega=\mathbf{000102211}$ and
  $\omega'=\mathbf{002002111}$. The corresponding groups do not have
  commensurable direct powers, at least if $A$ has odd order; indeed
  such a commensuration would have to be an isomorphism that restricts
  to an isomorphism from $G_{\overline\omega}$ to
  $G_{\overline{\omega'}}$, but these last groups are not isomorphic,
  see~\cite{grigorchuk:gdegree}.
\end{itemize}

These examples were obtained with the help of
Proposition~\ref{prop:sprad=>estimate}, by computing the
characteristic polynomials of the corresponding products of
matrices~\eqref{eq:matrices}.

\subsection{Growth of groups and wreath products}
We define a partial order on growth functions $v\colon\R_+\to\R_+$, as
follows. Say $v\precsim v'$ if, for some constants $C,D>0$, we have
$v(R)\le v'(CR)$ for all $R>D$. Then $v\sim v'$ if $v\precsim
v'\precsim v$.

We recall briefly the classical lemma that the equivalence class of a
growth function is independent of the generating set. More generally,
we shall consider \emph{weighted word metrics} of groups, given by a
weight function $\|\cdot\|\colon S\to\R_+^*$ for a generating set
$S$. The \emph{norm} of a word $w=s_1\dots s_\ell$ is then
$\|s_1\|+\cdots+\|s_\ell\|$, and the norm of $g\in G$ is the infimum
of norms of words representing it. The growth function is then
$v(R)=\#\{g\in G\mid R\ge\|g\|\}$.

Consider a group $A$, and a group $G$ acting on the right on a set
$X$. Their \emph{permutational wreath product} is $W=A\wr_X G:=(\sum_X
A)\rtimes G$, with the natural permutation action of $G$ on $\sum_X
A$. If $A,G$ are finitely generated and $G$ acts transitively on $X$,
then $W$ is finitely generated. Fix a basepoint $\rho\in X$; then $W$
is generated by $S\cup T$, with $S$ a generating set of $G$ and $T$
the set of functions $X\to A$ that vanish outside $\rho$ and take
values in a fixed generating set of $A$ at $\rho$.

The growth of $W$ is intimately related to the \emph{inverted orbit
  growth} of $G$ on $X$: the maximum $\Delta(R)$, over all words
$s_1\dots s_\ell\in S^*$ of norm at most $R$, of $\#\{\rho s_1\cdots
s_\ell,\rho s_2\cdots s_\ell,\dots,\rho s_\ell,\rho\}$. This was
already used in~\cite{bartholdi-erschler:permutational} to determine
the growth of $W_{\overline\zot}$, and
in~\cite{bartholdi-erschler:boundarygrowth} to prove that certain
permutational wreath products have exponential growth.

\subsection{Sketch of the argument}
Let us consider infinite sequences $\omega$ over the alphabet
$\{\mathbf0,\mathbf1,\mathbf2\}$ which alternates appropriately in
segments of the form $(\zot)^{i_k}$ and $\mathbf2^{j_k}$. For any such
sequence, Grigorchuk constructed in~\cite{grigorchuk:gdegree} a
$4$-generated group $G_\omega=\langle a,b,c,d\rangle$ of intermediate
growth, acting on the infinite binary rooted tree $\mathcal T$,
see~\S\ref{ss:grig}. The first generator permutes the two top-level
branches of $\mathcal T$, while the next three fix an infinite ray
$\rho$ in $\mathcal T$ and permute branches in the immediate
neighbourhood of $\rho$ according to $\omega$; the $i$th letter of
$\omega$ determines how $b,c,d$ act on the neighbourhood of the $i$th
point of $\rho$.

Let $X$ denote the orbit of $\rho$ under $G_\omega$, choose a
non-trivial finite group $A$, say $\Z/2\Z$, and define
$W_\omega=A\wr_X G_\omega$, the permutational wreath product of $A$
and $G_\omega$; see~\S\ref{ss:womega}.

We estimate the growth of $W_\omega$ by bounding from above the word
growth of $G_\omega$ and by bounding from above and below the inverted
orbit growth of $G_\omega$ on $X$, as
in~\cite{bartholdi-erschler:permutational}. The changes with respect
to~\cite{bartholdi-erschler:permutational} may be summarized as
follows: instead of considering a fixed group $G_\omega$, we consider
a sequence of groups $G_i$, and a dynamical system relating optimal
choices of metrics on them. We relate the growth of the $G_i$ to
properties of the dynamical system. The example of $W_{\overline\zot}$
is recovered as a $3$-periodic orbit in the dynamical system.

Let $G_i$ denote the group $G_{s^i\omega}$ for $s$ the shift map on
sequences. There are then injective homomorphisms $\phi_i\colon G_i\to
G_{i+1}\wr\sym_2$. For each $i$ we consider a space of metrics on
$G_i$, see~\S\ref{ss:metrics}. These metrics are specified by setting
the lengths of $a,b,c,d$, and are naturally parameterized by the open
$2$-simplex. The boundary of the simplex corresponds to degenerate
metrics, in which $\|a\|=0$, while the corners of the simplex
correspond to even more degenerate metrics, in which $\|a\|=\|x\|=0$
for some $x\in\{b,c,d\}$. The scale of the metric is set by the
condition $\|ab\|+\|ac\|+\|ad\|=1$.

We select for each $i$ a metric $\|\cdot\|_i$ on $G_i$ and a parameter
$\eta_i\in(2,3)$, such that the map $\phi_i$ is coarsely
$(2/\eta_i)$-Lipschitz (i.e.\ satisfies
$\|\phi_i(g)\|_{i+1}\le2/\eta_i\|g\|_i+C$ for a constant $C$) with
respect to the metric on $G_i$ and the $\ell_1$ metric on
$G_{i+1}\times G_{i+1}$; see Lemma~\ref{lem:contraction}. The optimal
choice of metrics is studied via a projective dynamical system on the
simplex. In fact, $\eta_i$ depends analytically on $\|\cdot\|_i$ and
the $i$th letter of $\omega$.  Whenever $\omega$ has long constant
subwords, the metrics degenerate towards the boundary of the simplex,
but do so in a controlled manner; in particular, we make sure that the
metrics never degenerate to a corner of the simplex.

We arrive at the heart of the argument. The first step is to deduce,
in Proposition~\ref{prop:gomega}, an upper bound on the word growth of
$G_i$, in terms of $\eta_1\cdots\eta_i$.

The next step, given in Proposition~\ref{prop:invgrowth}, deduces
sharp upper and lower bounds on the inverted orbit growth of $G_i$ on
$X$. The first two steps together give sharp bounds on the growth of
$W_\omega$, see Corollary~\ref{cor:global}.

For Theorem~\ref{thm:main}, the last step is to construct, out of a
function $f$ satisfying~\eqref{eq:gdoubling}, a sequence $\omega$ such
that $\eta_1\cdots\eta_i$ grows appropriately in relation to $f$; see
the precise statement in Lemma~\ref{lem:expconv} and~\eqref{eq:AB}.

For Theorem~\ref{thm:periodic}, the last step is to deduce, from
expansion properties of the dynamical system on the $2$-simplex, that
the averages $(\eta_1\cdots\eta_i)^{1/i}$ along periodic orbits of the
dynamical system take a dense set of values in $[2,\eta_+]$.

\subsection{Prior work}\label{ss:prior}
In 1983, Grigorchuk proved the existence of groups with intermediate
growth function between polynomial and exponential. This answered a
long-standing question asked by Milnor~\cite{milnor:5603}. Since then,
the upper bound on the growth of Grigorchuk's group
has been improved to
\[v_{G_{\overline\zot}}\precsim\exp(R^{\alpha_-})\approx\exp(R^{0.7674}),\]
see~\cite{bartholdi:upperbd}; and lower bounds of the form $\exp(R^\beta)$
have been found, but for $\beta$ substantially lower than $\alpha_-$;
see~\cites{muchnik-p:growth,leonov:lowerbd,brieussel:phd,bartholdi:lowerbd}.
Note that previous literature often uses the notation $\eta=2/\eta_+$.

Grigorchuk constructed in~\cite{grigorchuk:gdegree}*{Theorem~7.2} a
continuum of $2$-groups $G_\omega$ of intermediate growth, for
$\omega$ suitably chosen among infinite sequences over the alphabet
$\{\mathbf0,\mathbf1,\mathbf2\}$.  Similar constructions of
$p$-groups for any prime $p$, and of torsion-free groups were given
in~\cite{grigorchuk:pgps}.  He proved that many $G_\omega$ have
nonequivalent growth functions, that their growth functions are not
totally ordered, and that for every subexponentially growing function
$\rho$ there exists a group whose growth function $v$ satisfies
$v\not\precsim\rho$ and $v\prec\exp(R)$. This argument was modified
in~\cite{erschler:degrees} to show that for every such $\rho$ there
exists a group whose growth function satisfies $\rho\prec
v\prec\exp(R)$.

Grigorchuk's argument was to consider the groups $G_\omega$ for
$\omega$ containing very long subsequences of the form
$\mathbf2^{j_k}$; he showed that the growth of $G_\omega$ approaches
the exponential function, at least on a subsequence. For these same
sequences, we show that the growth of $W_\omega$ uniformly approaches
the exponential function.

At the other extreme, if $\omega=\overline\zot$, the growth of
$G_\omega$ is minimal among the known superpolynomial growth
functions.

Generally, good lower bounds for the growth of $G_\omega$ seem much
harder to construct than good upper bounds. On the one hand, there are
combinatorial constructions of sufficiently many elements in the ball
of radius $R$,
see~\cite{bartholdi:lowerbd,leonov:lowerbd,brieussel:phd};
probabilistic methods have also been used,
see~\cites{erschler:boundarysubexp,erschler:critical}.

A previous article by the
authors~\cite{bartholdi-erschler:permutational} explores the growth of
permutational wreath products of the form $A\wr_X G:=A^X\rtimes G$,
for a group $A$ and a group $G$ acting on a set $X$. They obtained
in this manner groups of intermediate growth $\sim\exp(R^\alpha)$ for
a sequence of $\alpha\to1$; those were the first examples of groups of
intermediate growth for which the growth function was determined. The
method of proof is an estimation of the growth of ``inverted orbits''
of $G$ on $X$, see~\S\ref{ss:inverted}.

This construction was used by Brieussel~\cite{brieussel:growth}, who
gave for every $\alpha\le\beta\in(\alpha_-,1)$ a group whose
growth function oscillates between $\exp(R^\alpha)$ and
$\exp(R^\beta)$ in the sense that $\log\log v(n)/\log(n)$ accumulates
both at $\alpha$ and $\beta$.

Kassabov and Pak~\cite{kassabov-pak:growth} construct, for
``sufficiently regular'' functions $v_{G_{\overline\zot}}\precsim
g_-\precsim g_+\precsim f_-\precsim f_+\precsim\exp(R)$, a group whose
growth function is between $g_-$ and $f_+$, and takes infinitely often
values in $[f_-,f_+]$ and in $[g_-,g_+]$.


Until recently, the analogue of Theorem~\ref{thm:main} wasn't even
known in the class of semigroups; though Trofimov shows
in~\cite{trofimov:growth} that for any $f_-\succnsim R^2$ and
$f_+\precnsim\exp(R)$ there exists a $2$-generated semigroup with
growth function infinitely often $\le f_-$ and infinitely often $\ge
f_+$. For every submultiplicative function $f$, a semigroup with
growth between $f(R)$ and $Rf(R)$ is constructed
in~\cite{bartholdi-smoktunowicz:images}. In particular, if $f(CR)\ge
Rf(R)$ for some $C>0$ and all $R\in\N$, then there exists a semigroup
with growth $\sim f$. In other words, all growth functions uniformly
above $R^{\log R}$ are realizable as the growth of a
semigroup. Warfield showed in~\cite{warfield:tensor} that for every
$\beta\in[2,\infty)$ there exists a semigroup with growth $\sim
R^\beta$. The only known gap in growth functions of semigroups is
Bergman's, between linear and
quadratic~\cite{krause-l:gkdim}*{Theorem~2.5}.

\subsection{Open problems}
We would very much like to obtain a complete description of which
equivalence classes of growth functions may occur as the growth of a
group. Thanks to Gromov's result~\cite{gromov:nilpotent} that groups
of polynomial growth are virtually nilpotent and therefore of growth
$\sim n^d$ for an integer $d$, we are interested in groups of
intermediate growth. The groups that we construct have growth above
$\exp(R^{\alpha_-})$. A tantalizing open problem in the existence of
groups of growth between polynomial and $\exp(\sqrt n)$; a conjecture
of Grigorchuk asserts that they do not exist. It is not even known
whether there are groups whose growth is strictly between polynomial
and $\exp(R^{\alpha_-})$.




\section{Self-similar groups}\label{ss:ss}
Self-similar groups are groups endowed with a \emph{self-similar}
action on sequences, namely, an action that is determined by actions
on subsequences. We adopt a more algebraic notation.

\begin{definition}\label{defn:ss}
  A \emph{self-similar group} $G$ is a group endowed with a
  homomorphism $\phi:G\to G\wr\sym_d$, for some $d\in\N$.

  A \emph{self-similar group sequence} is a sequence $(G_0,G_1,\dots)$
  of groups, with embeddings $\phi_i:G_i\to G_{i+1}\wr\sym_{d_i}$ for
  all $i\in\N$.
\end{definition}

A self-similar group gives rise to a self-similar group sequence: if
$\phi:G\to G\wr\sym_d$, then set $G_i=G$, $d_i=d$ and $\phi_i=\phi$
for all $i\in\N$.

Let $(G_i)$ be a self-similar group sequence, the $d_i$ and $\phi_i$
being implicit in the notation. Define
\[T_0=\bigsqcup_{i\in\N}\{1,\dots,d_0\}\times\cdots\times\{1,\dots,d_{i-1}\}.
\]
This is the vertex set of a rooted tree (with root the empty product
$\emptyset$), if one puts an edge between $x_0\dots x_{i-1}$ and
$x_0\dots x_{i-1}x_i$ for all $x_j\in\{1,\dots,d_j\}$. We denote also
this tree by $T_0$.

The group $G_0$ acts by isometries on $T_0$ as follows. Given $g\in G$
with $\phi_0(g)=(f,\pi)$ and $x_0\dots x_{i-1}\in T_0$, set
\[(x_0\dots x_{i-1})^g=\begin{cases}\emptyset & \text{ if }i=0,\\
  (x_0^\pi)\,(x_1\dots x_{i-1})^{f(x_0)} & \text{ inductively if }i>0.\end{cases}
\]

The boundary of $T_0$ is naturally identified with the set of infinite sequences
\[\partial T_0=\prod_{i\in\N}\{1,\dots,d_i\},\]
endowed with the product topology, and the action of $G$ on $T_0$
extends to a continuous action on $\partial T_0$.

\begin{definition}
  The self-similar group sequence $(G_i)$ is \emph{branched} if there
  exist for all $i\in\N$ a finite-index subgroup $K_i\le G_i$ such
  that $K_{i+1}^{d_i}\le\phi_i(K_i)$ for all $i\in\N$.
\end{definition}
For more information on branch groups, and the algebraic consequences
that follow from that property, we refer to~\cite{bartholdi-g-s:bg}.

\subsection{The Grigorchuk groups \boldmath $G_\omega$}\label{ss:grig}
Our main examples were constructed by Grigorchuk
in~\cite{grigorchuk:gdegree}. Let $\{\mathbf0,\mathbf1,\mathbf2\}$
denote the three non-trivial homomorphisms from the four-group
$\{1,b,c,d\}$ to the cyclic group $\{1,a\}$, ordered for definiteness
by the condition that, in that order, they vanish on $b,c,d$
respectively. Write $\Omega=\{\mathbf0,\mathbf1,\mathbf2\}^\N$ the
space of infinite sequences over $\{\mathbf0,\mathbf1,\mathbf2\}$,
endowed with the shift map $s\colon\Omega\to\Omega$ given by
$s(\omega_0\omega_1\dots)=\omega_1\omega_2\dots$.

Fix a sequence $\omega\in\Omega$, and define a self-similar group
sequence as follows. Take $d=2$, and write $\sym_d=\{1,a\}$. Each
$G_i$ is generated by $\{a,b,c,d\}$. Set
\[\phi_i(a)=\pair{1,1}a,\quad\phi_i(b)=\pair{\omega_i(b),b},\quad\phi_i(c)=\pair{\omega_i(c),c},\quad\phi_i(d)=\pair{\omega_i(d),d}.\]
The assertion that the $\phi_i$ are homomorphisms, and that the $G_i$
act faithfully on the binary rooted tree, define the groups $G_i$
uniquely.

Define the \emph{Grigorchuk group} $G_\omega$ as the group $G_0$
constructed above. We then have $G_i=G_{s^i\omega}$ for all $i\ge0$,
so there are homomorphisms $\phi_\omega:G_\omega\to
G_{s\omega}\wr\sym_2$, and $G_0$ is the first entry of a self-similar
group sequence.


If $\omega$ is not ultimately constant, it is known that each of the
$G_\omega$ have intermediate growth, with growth bounded from below by
$\exp(R^{1/2})$ and from above by $\exp(R)$; and no smaller function
may serve as upper bound~\cite{grigorchuk:gdegree}*{Theorem~7.1}.

(This may be also seen concretely as follows: consider the ``universal''
group $\widehat G$, defined as the diagonal subgroup generated by
$\hat a=(a,a,\dots),\hat b=(b,b,\cdots),\hat c=(c,c,\dots),\hat
d=(d,d,\dots)$ in $\prod_{\omega\in\Omega}G_\omega$. Then this group
has exponential growth --- indeed, $\{\hat a\hat b,\hat a\hat c,\hat
a\hat d\}$ freely generates a free semigroup.)

\begin{proposition}
  The group $G_\omega$ is the first term of a self-similar group
  sequence;
  \begin{enumerate}
  \item if $\omega$ contains only a finite number of one of the three
    symbols, then $G_\omega$ contains an element of infinite order;
  \item if the sequence $\omega$ contains infinitely many of each of
    the three symbols, then $G_\omega$ is an infinite, branched,
    torsion group;
  \item if $\omega$ is periodic and contains each of the three symbols
    $\mathbf0,\mathbf1,\mathbf2$, then $G_\omega$ is also self-similar.
  \end{enumerate}
\end{proposition}
\begin{proof}
  It was already proven in~\cite{grigorchuk:gdegree} that $G_\omega$
  is torsion if and only if $\omega$ contains all three symbols.

  To show that $G_\omega$ are branched, consider the groups
  $G_i=G_{s^i\omega}$ and homomorphisms $\phi_i:G_i\to
  G_{i+1}\wr\sym_2$. For each $i\in\N$, let $x_i\in\{b,c,d\}$ be such
  that $\omega_i(x_i)=1$, and set
  $K_i=\langle[x_i,a]\rangle^{G_i}$. Choose also
  $y_i\in\{b,c,d\}\setminus\{x_i\}$. Then $1\times K_{i+1}$ is
  normally generated by $\pair{1,[x_{i+1},a]}$, and
  \[\pair{1,[x_{i+1},a]}=\phi_i([x_i,y_i^a])=\phi_i([x_i,a][x_i,a]^{ay_i})\in \phi_i(K_i);\]
  the same holds for $\pair{[x_{i+1},a],1}$.

  Consider first $G_i/\langle x_i\rangle^{G_i}$. This group is
  generated by two involutions $a$ and $y_i$, so is a finite dihedral
  group, because $G_i$ is torsion. It follows that $\langle
  x_i\rangle^{G_i}$ has finite index in $G_i$. Then $\langle
  x_i\rangle^{G_i}/K_i=\langle x_i\rangle K_i$ has order $2$, so $K_i$
  also has finite index.

  Finally, if $\omega$ is periodic, say of period $n$, then $G_0=G_n$
  and the composition of the maps $\phi_i$ gives a homomorphism
  $\phi:G_0\to G_0\wr\sym_{2^n}$, showing that $G_\omega$ is
  self-similar.
\end{proof}

\subsection{The Groups \boldmath $W_\omega$}\label{ss:womega}
Fix a finitely generated group $A$, an infinite sequence
$\rho\in\{1,2\}^\infty$ such as for example $\rho=2^\infty$, and its orbit
$X:=\rho\cdot G_\omega$.  Consider $W_\omega:=A\wr_X G_\omega$.

\begin{proposition}\label{prop:Wbranched}
  The group $W_\omega$ is finitely generated, and is the first term of
  a branched self-similar group sequence.

  If $\omega$ and $\rho$ are periodic, then $W_\omega$ is also
  self-similar.

  If $\omega$ and $\rho$ are recursive and $A$ is recursively
  presented, then so is $W_\omega$. If moreover $A$ has a solvable
  word problem, then so does $W_\omega$.

  If $A$ is torsion and $\omega$ contains infinitely many of all three
  symbols, then $W_\omega$ is torsion.
\end{proposition}
\begin{proof}
  Clearly $W_\omega$ is generated by $\{a,b,c,d\}\cup S$ for a finite
  generating set $S$ of $A$; we imbed $A$ in $W_\omega$ as those
  functions $X\to A$ supported on $\{\rho\}$.

  Let $(G_i)$ be the self-similar group sequence with
  $G_0=G_\omega$. Let $\rho_i$ be the subsequence of $\rho$ starting
  at position $i$.  Define $W_i=A\wr_{\rho_iG_i}G_i$, with imbeddings
  $\psi_i:W_i\to W_{i+1}\wr\sym_2$ given as follows. Consider
  $w=(u,g)\in W_i$, with $u\colon\rho_iG\to A$ and $g\in G_i$. Write
  $\phi_i(g)=(f,\pi)$ with $f\colon\{1,2\}\to G_{i+1}$ and
  $\pi\in\sym_2$. For $k\in\{1,2\}$, set $u_k(x):=u(k
  x)\colon\rho_{i+1}G_{i+1}\to A$. Then
  \[\psi_i(u,g)=(k\mapsto (u_k,f(k)),\pi).\]

  Furthermore, let $K_i$ be the finite-index subgroup of $G_i$ for
  which $(G_i)$ is branched. Then $L_i:=A\wr K_i$ has finite index in
  $W_i$, and shows that $(W_i)$ is branched.

  If $\omega$ is periodic, say of period $n$, then $G_0=G_n$ and we
  get, composing the maps $\phi_0,\dots,\phi_{n-1}$, an imbedding
  $\phi:G_0\to G_0\wr\sym_{2^n}$, showing that $G_\omega$ is
  self-similar. The same holds for $W_\omega$.

  For a presentation of $W_\omega$, see~\cite{bartholdi:lpres},
  \cite{cornulier:fpwreath}
  and~\cite{bartholdi-erschler:permutational}. Let $A$ be generated by
  a finite set $S$; then $W_\omega$ is generated by
  $S\sqcup\{a,b,c,d\}$. The relations of $W_\omega$ are the following:
  those of $A$; those of $G_\omega$; commutation relations
  $[s,(s')^g]$ for all $s,s'\in S$ and $g\in G_\omega$ not fixing
  $\rho$; and commutation relations $[s,g]$ for all $s\in S$ and $g\in
  G_\omega$ fixing $\rho$.

  If $\omega$ is recursive, then the action of $G_\omega$ on $X$ is
  computable, so the sets $\{w\in\{a,b,c,d\}^*\mid \rho w=w\}$ and
  $\{w\in\{a,b,c,d\}^*\mid \rho w\neq w\}$ are recursive; so
  $W_\omega$ is recursively presented.

  If moreover the word problem is solvable in $A$, then it is also
  solvable in $W_\omega$. Indeed $G_\omega$ is contracting, so the
  word problem is solvable in $G_\omega$;
  see~\S\ref{ss:presentations}. A word in $W_\omega$ is trivial if and
  only if its image in $G_\omega$ is trivial and its values at all
  $x\in X$ are trivial in $A$.

  If $\omega$ contains infinitely many of all three symbols, then
  $G_\omega$ is torsion; if moreover $A$ is torsion, then $W_\omega$,
  being an extension of two torsion groups, is also torsion.
\end{proof}

\subsection{The torsion-free groups \boldmath $\tilde G_\omega$}\label{ss:torsionfree}
Grigorchuk constructed in~\cite{grigorchuk:pgps}*{\S5} a torsion-free
group $\tilde G$ of intermediate growth; it was later noted
in~\cite{grigorchuk-m:interm} that $\tilde G$ acts continuously on the
interval $[0,1]$. We shall consider variants $\tilde G_\omega$ of his
construction, which we now recall.

The group $\tilde G_\omega$ is the first term of a self-similar group
sequence in the sense of Definition~\ref{defn:ss}, but the map
$\phi:\tilde G_\omega\to\tilde G_{s\omega}\wr\sym_2$ is not
injective. We construct the groups $\tilde G_\omega$ via a slightly
different definition. Each of the groups $\tilde G_\omega$ acts
faithfully on the rooted tree of infinite \arity, whose vertex set is
naturally identified with finite sequences over $\Z$, and whose
boundary is naturally the space $\Z^\infty$ of infinite sequences of
integers. Each $\tilde G_\omega$ is generated by $\{a,b,c,d\}$. Each
$\omega\in\{\mathbf0,\mathbf1,\mathbf2\}$ is now understood as a
homomorphism $\langle b,c,d\rangle\cong\Z^3\to\langle
a\rangle\cong\Z$, via $\mathbf0(b)=1$, $\mathbf0(c)=\mathbf0(d)=a$ etc.\
under cyclic permutation.  There is a homomorphism
$\tilde\phi_\omega\colon\tilde G_\omega\to\prod_\Z\tilde
G_\omega\rtimes\Z$, given by
\begin{align*}
  \tilde\phi_\omega(a) &= \pair{\dots,1,1,\dots}(n\mapsto n+1),\\
  \tilde\phi_\omega(x) &= \pair{\dots,\omega_0(x),x,\dots}\text{ for all }x\in\langle b,c,d\rangle,
\end{align*}
where the $\dots$ indicate that the sequence repeats
$2$-periodically. In particular, the map $a\mapsto a,b\mapsto
b,c\mapsto c,d\mapsto d$ extends to an epimorphism $\xi_\omega\colon\tilde
G_\omega\to G_\omega$.

The kernel of $\xi_{\overline\zot}$ is abelian, isomorphic to
$\Z\times\sum_X\Z^3$; this is essentially due to Grigorchuk,
see~\cite{bartholdi-erschler:permutational}*{Lemma~5.7} for a
proof. Note then that the argument extends with no changes to
$\ker\xi_\omega$ for arbitrary $\omega$. We proved in
\cite{bartholdi-erschler:permutational}*{Proposition~5.8} that $\tilde
G_{\overline\zot}$ has growth function equivalent to $\Z^3\wr
G_{\overline\zot}$; again, the argument extends without changes to all
$\tilde G_\omega$.

The following result essentially appears
in~\cite{grigorchuk-m:interm}; there, Grigorchuk and Mach\`\i\ show
that $\tilde G_{\overline\zot}$ is residually virtually nilpotent and
residually solvable.
\begin{lemma}\label{lem:tildegrf}
  The group $\tilde G_\omega$ is residually $2$-finite.
\end{lemma}
\begin{proof}
  For all $n\in\N$, let $\pi_n$ denote the homomorphism
  $G_\omega\to\sym_{2^n}$ defined by restricting the action of
  $G_\omega$ to sequences in $\{1,2\}^n$, and let $N_n$ be the normal
  closure in $\tilde G_\omega$ of $\langle
  a^{2^n},b^{2^n},c^{2^n},d^{2^n}\rangle$.

  Consider then the normal subgroups
  $P_n=N_n\ker(\pi_n\circ\xi_\omega)$ of $\tilde G_\omega$. On the one
  hand, $\tilde G_\omega/P_n$ is a finite $2$-group: it is a torsion
  abelian $2$-extension of the finite $2$-group $\pi_n(G_\omega)$. On
  the other hand, $P_\infty:=\bigcap_{n\in\N}P_n$ is trivial:
  $P_\infty$ is contained in $\ker\xi_\omega\cong\Z\times\sum_X\Z^3$,
  and is $2$-divisible in $\tilde G_\omega$.
\end{proof}

\subsection{Presentations}\label{ss:presentations}
Assume that $\omega$ is a recursive sequence; that is, there is an
algorithm that, given $n\in\N$ as input, computes the $n$th letter of
$\omega$. We then claim that $G_\omega$ and $\tilde G_\omega$ have
solvable word problem; that is, there are algorithms that, given a
word $w\in\{a,b,c,d\}^*$ as input, compute whether $w=1$ in
$G_\omega$ and whether $w=1$ in $\tilde G_\omega$.

Indeed, the following algorithm solves the word problem in
$G_\omega$. Given $w\in\{a,b,c,d\}^*$, first perform all elementary
cancellations $a^2\to1,b^2\to1,\dots,bc\to d,\dots$. If $w$ became the
empty word, then $w=1$ in $G_\omega$. Otherwise, if the exponent sum
of $a$ in $w$ is non-zero, then $w\neq1$ in $G_\omega$. Otherwise,
compute $\omega_0$, and using it compute
$\phi_0(w)=\pair{w_1,w_2}$. Noting that $s\omega$ is again a recursive
sequence, apply the algorithm recursively to $w_1$ and $w_2$, to
determine whether both are trivial in $G_{s\omega}$.

The algorithm terminates because $G_\omega$ is contracting.

A small modification of that algorithm solves the word problem in
$\tilde G_\omega$. Given
$w\in\{a^{\pm1},b^{\pm1},c^{\pm1},d^{\pm1}\}^*$, first perform all
elementary substitutions $cb\to bc,dc\to cd,db\to bd$ and free
cancellations. If $w$ became the empty word, then $w=1$ in $\tilde
G_\omega$. Otherwise, if the exponent sum of $a$ in $w$ is non-zero,
or if $w\in\langle b,c,d\rangle$, then $w\neq1$ in $\tilde
G_\omega$. Otherwise, compute $\omega_0$, and using it compute
$\tilde\phi_0(w)=\pair{w_1,w_2}$. Noting that $s\omega$ is again a
recursive sequence, apply the algorithm recursively to $w_1$ and
$w_2$, to determine whether both are trivial in $\tilde G_{s\omega}$.

\begin{corollary}
  If $\omega$ is recursive, then $G_\omega$ and $\tilde G_\omega$ have
  solvable word problem, and are recursively presented.
\end{corollary}
\begin{proof}
  The solution of the word problem is given by the algorithms
  above. The groups are then recursively presented, by taking as
  relators, e.g., all words that are trivial in them.
\end{proof}
More efficient recursive presentations are described
in~\cite{sidki:presentation}.

Note that (unless $\omega$ is eventually constant) the group
$G_\omega$ is not finitely presented. In fact, no finitely presented
example of branched group with injective $\phi$ is known.

On the other hand, contracting groups may be described by a finite
amount of data, and this immediately implies that there are countably
many such groups.  Let $G$ be a contracting group, with
self-similarity map $\phi:G\hookrightarrow G\wr\sym_d$. Fix the
notation $\phi(g)=\pair{g_1,\dots,g_d}\pi_g$. The contracting property
implies that there are constants $\lambda<1$ and $C$ such that
$\|g_i\|\le\lambda\|g\|+C$.

Fix a generating set $S$ for $G$. The map $\phi$ is determined by its
values on generators, namely by $\#S$ permutations in $\sym_d$ and
$d\#S$ words in the free group on $S$. The corresponding map
$\phi:F_S\to F_S\wr\sym_d$ is not contracting, but there exists a
finitely presented quotient $\widehat G$ of $F_S$, with
$F_S\twoheadrightarrow\widehat G\twoheadrightarrow G$, for which the
induced map $\phi:\widehat G\to\widehat G\wr\sym_d$ is contracting
(but not injective).

The contraction property implies the following: for any $g\in \widehat
G$, sufficiently many applications of $\phi$ and projection to a
\coordinate\ $g_i$ result in elements of norm at most
$B:=(C+1)/(1-\lambda)$. There are finitely many elements in the ball
of radius $B$ in $\widehat G$; let us fix which of these elements
project to the identity in $G$.

We claim that these data --- the map $\phi:S\to F_S\wr\sym_d$, the
finitely presented group $\widehat G$, and the set $W$ of words of
length less than $B$ in $G$ projecting to the identity in $G$ ---
determine $G$ completely. To see that, it is sufficient to exhibit an
algorithm solving the word problem in $G$.

Given $g\in G$, written as a word in $S$, do the following. If
$\|g\|<B$, then $g=1$ if and only if it belongs to $W$. Otherwise,
compute $\phi(g)=\pair{g_1,\dots,g_d}\pi_g$. If $\pi_g\neq1$, then
$g\neq1$. Otherwise, apply recursively the algorithm to determine
whether $g_1,\dots,g_d$ are all trivial in $G$. This is well-founded
because (by the contraction property) we have $\|g_i\|\le \|g\|-1$.

\section{Metrics on the Grigorchuk groups $G_\omega$}\label{ss:metrics}
Let $G_\omega$ be a Grigorchuk group as above; it is the first entry
of a self-similar group sequence with $G_i=G_{s^i\omega}$.  We put a
word metric on each $G_i$, i.e.\ a metric $\|\cdot\|_i$ defined by
assigning a norm to each generator and extending it naturally to
group elements. This metric is determined by a point $V_i$ in the open
simplex
\[\Delta=\big\{(\beta,\gamma,\delta)\mid\max\{\beta,\gamma,\delta\}<\frac12,\beta+\gamma+\delta=1\big\}\]
with vertices $(\frac12,\frac12,0)$, $(\frac12,0,\frac12)$,
$(0,\frac12,\frac12)$.  It assigns norm
$\alpha:=1-2\max\{\beta,\gamma,\delta\}$ to $a$, and norms
$\beta-\alpha,\gamma-\alpha,\delta-\alpha$ to $b,c,d$ respectively.

In particular, the norm of the longest of $b,c,d$ equals the sum of
the norms of the two shortest.

In order to appropriately construct metrics on the $G_i$, we consider
the space $\Delta\times\Omega$, and define three $3\times3$
matrices $M_x$, for $x\in\{\mathbf0,\mathbf1,\mathbf2\}$, by
\begin{equation}\label{eq:matrices}
  M_{\mathbf0}=\begin{pmatrix}1&1&1\\0&2&0\\0&0&2\end{pmatrix},\\
  M_{\mathbf1}=\begin{pmatrix}2&0&0\\1&1&1\\0&0&2\end{pmatrix},\\
  M_{\mathbf2}=\begin{pmatrix}2&0&0\\0&2&0\\1&1&1\end{pmatrix}.
\end{equation}
Define a continuous function
$\eta\colon\Delta\times\{\mathbf0,\mathbf1,\mathbf2\}\to(2,3]$ by the
condition $\eta(p,x)^{-1}M_x(p)\in\Delta$, and projective maps
$\overline M_x\colon\Delta\to\Delta$ by the $\overline
M_x(p,x)=\eta(p,x)^{-1}M_x(p)$. Define next a shift map
$s\colon\Delta\times\Omega\to\Delta\times\Omega$ by
\[s(p,\omega_0\omega_1\dots)=(\overline
M_{\omega_0}(p),\omega_1\omega_2\dots).
\]
Define finally a continuous function $\mu\colon\Delta\to(0,\frac13]$
by
\[\mu(\beta,\gamma,\delta)=\min\{\beta,\gamma,\delta\}.\]
The function $\mu(p)$ is equivalent to the minimal euclidean distance
from $p\in\Delta$ to a vertex of $\Delta$.

Note that the $\overline M_x$ have disjoint images, and that the union
of their images is dense in $\Delta$. More precisely, each $\overline
M_x$ maps $\Delta$ to the simplex spanned by two vertices and the
barycentre of $\Delta$. The shift map
$s\colon\Delta\times\Omega\to\Delta\times\Omega$ is injective with dense
image. Define $\Delta'\subset\Delta$ by
\[\Omega^+=\bigcap_{k\ge0}s^k(\Delta\times\Omega)=\Delta'\times\Omega.\]
Then $\Delta'$ is dense in $\Delta$, being the complement of countably
many line segments, and $s\colon\Omega^+\to\Omega^+$ is bijective. The
metric on $G_i$ determined by $p\in\Delta'$ is non-degenerate.

Endow $\Delta$ with the Hilbert metric
$d_\Delta(V_1,V_2)=\log(V_1,V_2;V_-,V_+)$, computed using the
cross-ratio of the points $V_1,V_2$ and the intersections $V_-,V_+$ of
the line containing $V_1,V_2$ with the boundary of $\Delta$. Endow
$\Omega$ with the Bernoulli metric
$d(\omega,\omega')=\sum_{i\colon\omega_i\neq\omega'_i}2^{-i}$.

\begin{lemma}[Essentially~\cite{birkhoff:jentzsch}; we have not been able to locate a proof of the following extension]\label{lem:birkhoff}
  Let $K$ be a convex subset of a real vector space, and let $A:K\to
  K$ be a projective map. Then $A$ is $1$-Lipschitz for the Hilbert metric.

  If furthermore $A(K)$ contains no lines from $K$ (that is, $A(K)\cap
  \ell\neq K\cap \ell$ for every line $\ell$ intersecting $K$), then
  $A$ is strictly contracting.
\end{lemma}
\begin{proof}
  Since $A$ is projective, it preserves the cross-ratio on lines, so
  we have $d_{A(K)}(A(V_1),A(V_2))=d_K(V_1,V_2)$ for all $V_1,V_2\in
  K$. Furthermore, on the line $\ell$ through $V_1,V_2$, the
  intersection points $\ell\cap\partial A(K)$ are not further from
  $V_1,V_2$ than $\{V_+,V_-\}=\ell\cap\partial K$; the Hilbert metric
  decreases as $V_\pm$ are moved apart from $V_1,V_2$, and this gives
  strict contraction under the condition $A(K)\cap \ell\neq K\cap
  \ell$.
\end{proof}

\begin{lemma}
  The homeomorphism $s\colon\Omega^+\to\Omega^+$ is `hyperbolic': it
  decomposes as the product of an expanding and contracting map.
\end{lemma}
\begin{proof} The expanding direction is $\Omega$; indeed the
  Bernoulli metric is doubled by the shift map. The contracting
  direction is $\Delta$; indeed the maps $\overline M_x$ are
  projective, hence contracting by Lemma~\ref{lem:birkhoff}. Note that
  the contraction is strict, but not uniform.
\end{proof}

Let $V_0\in\Delta'$ be a given metric on $G_\omega$; this defines a
point $(V_0,\omega)\in\Omega^+$. We then let $V_i$ be the first
\coordinate\ of $s^i(V_0,\omega)$, and have defined a metric
$\|\cdot\|_i$ on each $G_i$. We also set $\eta_i=\eta(V_i,\omega_i)$
and $\mu_i=\mu(V_i)$, to lighten notation.

(Note that the choice of $V_0\in\Delta'$ will ultimately not be very
important, because the maps $\overline M_x$ are contracting. Indeed,
for two choices $V_0,V_0'$ we have $V_i\to V_i'$, so $\mu_i\to\mu_i'$
and $\eta_i\to\eta_i'$ as $i\to\infty$.)

\begin{lemma}\label{lem:contraction}
  For each $i\ge0$ and each $g\in G_i$, with
  $\phi_i(g)=\pair{g_1,g_2}\pi$, we have
  \[\|g_1\|_{i+1}+\|g_2\|_{i+1}\le\frac2{\eta_i}(\|g\|_i+\|a\|_i).\]
\end{lemma}
\begin{proof}
  We consider first $\omega_i=\mathbf0$; the general case will
  follow. Consider $V_i=(\beta_i,\gamma_i,\delta_i)\in\Delta$, and
  $V_{i+1}=\overline
  M_{\mathbf0}(V_i)=(\beta_{i+1},\gamma_{i+1},\delta_{i+1})$. We
  compute $\eta_i=\eta(V_i,\mathbf0)=3-2\beta_i$, and
  $V_{i+1}=(1/\eta_i,2\gamma_i/\eta_i,2\delta_i/\eta_i)$, so
  \[\|b\|_{i+1}=\beta_{i+1}-\|a\|_{i+1}=\beta_{i+1}-1+2\beta_{i+1}=3/\eta_i-1=2\beta_i/\eta_i.\]
  The same result holds, cyclically permuting $\beta,\gamma,\delta$,
  for all $\overline M_x$. We therefore have
  \begin{align}
    \|b\|_{i+1}+\|\omega_i(b)\|_{i+1}&= 2\beta_i/\eta_i = 2/\eta_i(\|b\|_i+\|a\|_i),\notag\\
    \|c\|_{i+1}+\|\omega_i(c)\|_{i+1}&= 2\gamma_i/\eta_i = 2/\eta_i(\|c\|_i+\|a\|_i),\label{eq:bcd}\\
    \|d\|_{i+1}+\|\omega_i(d)\|_{i+1}&= 2\delta_i/\eta_i = 2/\eta_i(\|d\|_i+\|a\|_i).\notag
  \end{align}

  On the other hand, $g$ may by written as an alternating word in $a$
  and $b,c,d$, because of the relations $b^2=c^2=d^2=bcd=1$. Group
  each $b,c,d$-letter with an $a$, and sum the corresponding
  inequalities; there may be an $a$ left over. This gives the desired
  inequality.
\end{proof}

Even though the metrics we consider on $G_\omega$ are not word
metrics, the volume of small balls are well understood. We include the
following result for future reference, though it will not be used in
this article:
\begin{lemma}
  There are two absolute constants $K_1,K_2>0$ such that, for every
  group $G_\omega$ with metric given by $V_0\in\Delta$, we have
  \[\frac{K_1}{\mu_0}\le v(1/2)\le\frac{K_2}{\mu_0}.\]
\end{lemma}
\begin{proof}
  Assume without loss of generality $\beta\le\gamma\le\delta$.  A word
  of norm $\le1$ and length $\ge4$ may contain only $a$s and $b$s,
  because otherwise it would have norm at least $\beta+\gamma>1$. It
  must therefore be of the form $b^i(ab)^ja^k$ for $i,k\in\{0,1\}$ and
  $0\le j\le\beta^{-1}$.
\end{proof}

\begin{lemma}\label{lem:growth0}
  For every $A>0$ there exists a constant $K$, such that, for all
  $V_0\in\Delta'$, we have $v(A\mu_0)\le K$.
\end{lemma}
\begin{proof}
  The minimal norm in $(G_\omega,\|\cdot\|_{V_0})$ of a non-trivial
  product of two generators is $\mu_0$; we can therefore take
  $K=4^{2A}$.
\end{proof}

\subsection{Volume growth of \boldmath $G_\omega$, upper bound}\label{ss:growthupper}
Let $v_i(R)$ denote the volume growth of $G_i$ for the metric
$\|\cdot\|_i$. The following proposition, and its variants
in~\S\ref{ss:inverted}, are key to our computation of growth
functions.

\begin{proposition}\label{prop:gomega}
  There is an absolute constant $B$ such that for all $k\in\N$, we
  have
  \[v_0(\eta_0\eta_1\cdots\eta_{k-1}\mu_k)\le B^{2^k}.\]
\end{proposition}

Before embarking in the proof, we give a few preliminary definitions
and results.  For a non-negative function $f(R)$, its \emph{concave
  majorand} is the smallest concave function $f^+$ greater than $f$;
the graph of $f^+$ bounds the convex hull of the graph of $f$. It may
be defined by
\begin{equation}\label{eq:majorand}
  f^+(R)=\sup_{p<R<q}\left(\frac{q-R}{q-p}f(p) + \frac{R-p}{q-p}f(q)\right).
\end{equation}

The following lemma is not used in the argument, but only included for
illustration. It implies that the growth functions considered in
Theorem~\ref{thm:main} are equivalent to log-concave functions.
\begin{lemma}\label{lem:concave}
  Let $g\colon\R\to\R$ be a function satisfying
  \begin{equation}\label{eq:g0}
    g(2R)\le 2g(R)\le g(\eta_+R)
  \end{equation}
  for all $R$ large enough. Then $g$ is equivalent to a concave
  function.
\end{lemma}
\begin{proof}
  First, the inequality $g(2R)\le g(R)$ alone implies that $g$ is
  equivalent to a subadditive function, namely a function satisfying
  $g(R+S)\le g(R)+g(S)$. We therefore assume, without loss of
  generality, that $g$ is subadditive.  Let $g^*$ be the concave
  majorand of $g$, as in~\eqref{eq:majorand}. We will show $g^*\le
  2g$; this with the second inequality in~\eqref{eq:g0} yields the
  claim.

  Without loss of generality, assume also that $g$ is increasing.
  Consider $a<x<b$ such that $g^*(x)=(x-a)/(b-a) g(b) + (b-x)/(b-a)
  g(a)$. Write $b=Nx+c$ with $0\le c<x$. Then
  \begin{align*}
    (b-a) g^*(x) &= (x-a)g(Nx+c) + ((N-1)x+c) g(a)\\
    &\le (x-a)(Ng(x)+g(c)) + ((N-1)x+c)g(a)\\
    &\le (x-a)(N+1)g(x) + ((N-1)x+c)g(x)\\
    &= (2Nx+c-(N+1)a)g(x)\le 2(b-a)g(x),
  \end{align*}
  so $g^*(x)\le 2g(x)$.
\end{proof}

\begin{lemma}\label{lem:etamu}
  For all $k\in\N$ we have $\eta_k\mu_{k+1}\ge\mu_k+\|a\|_k.$
\end{lemma}
\begin{proof}
  We write $V_k=(\beta,\gamma,\delta)$, and assume without loss of
  generality $\omega_k=\mathbf0$. Then the inequality
  \[\eta_k\mu_{k+1}=2\min\{\gamma,\delta\}\ge\min\{\beta,\gamma,\delta\}+(1-\max\{\beta,\gamma,\delta\})=\mu_k+\|a\|_k\]
  is checked by considering the six cases $\beta<\gamma<\delta$ etc. in turn.
\end{proof}

\begin{proof}[Proof of Proposition~\ref{prop:gomega}]
  We derive a relation between $v_i$ and $v_{i+1}$ from
  Lemma~\ref{lem:contraction}. The integrals $\int_{p_1+p_2=K}$ are
  really finite sums, because their arguments are piecewise constant;
  however, in these finite sums the parameters $p_1,p_2$ do not take
  integral values. We have:
  \begin{align}
    v_i(R) &\le\int_{p_1+p_2=2/\eta_i(R+\|a\|_i)\kern-3em}2v_{i+1}(p_1)v_{i+1}(p_2)\label{eq:contr}\\\
    &\le2\big(1+\frac2{\eta_i}(R+\|a\|_i)\big)\max_{p_1+p_2=2/\eta_i(R+\|a\|_i)}v_{i+1}(p_1)v_{i+1}(p_2).\notag
  \end{align}
  Set $\alpha_i(R):=(\log v_i)^+(R)$. Below, we abbreviate in `$\cdots$'
  the development of $\log v_i(q)$, which is similar to that of $\log
  v_i(p)$:
  \begin{align*}
    \alpha_i(R) &=\sup_{p<R<q}\left\{\frac{q-R}{q-p}\log v_i(p)+\frac{R-p}{q-p}\log v_i(q)\right\}\\
    &\le\sup_{p<R<q}\left\{\frac{q-R}{q-p}\log\int_{p_1+p_2=\frac2{\eta_i}(p+\|a\|_i)\kern-3em}2v_{i+1}(p_1)v_{i+1}(p_2)+\frac{R-p}{q-p}\cdots\right\}\text{ by }\eqref{eq:contr}\\
    &\le\log2+\sup_{p<R<q}\left\{\frac{q-R}{q-p}\log\int_{p_1+p_2=\frac2{\eta_i}(p+\|a\|_i)\kern-3em}\exp\big(\alpha_{i+1}(p_1)+\alpha_{i+1}(p_2)\big)+\frac{R-p}{q-p}\cdots\right\}\\
    &\begin{aligned}
      \le\log2+\sup_{p<R<q}&\left\{\frac{q-R}{q-p}\log\int_{p_1+p_2=\frac2{\eta_i}(p+\|a\|_i)}\exp\bigg(2\alpha_{i+1}\Big(\frac{p+\|a\|_i}{\eta_i}\Big)\bigg)\right.\\
      &\qquad\left.+\frac{R-p}{q-p}\cdots\right\}\text{ because $\alpha_{i+1}$ is concave}
    \end{aligned}\\
    &\le\log2+\sup_{p<R<q}\left\{\frac{q-R}{q-p}\left[\log\Big(1+\frac{2(p+\|a\|_i)}{\eta_i}\Big)+2\alpha_{i+1}\Big(\frac{p+\|a\|_i}{\eta_i}\Big)\right]+\frac{R-p}{q-p}\cdots\right\}\\
    &\le\log2+\log\Big(1+\frac{2(R+\|a\|_i)}{\eta_i}\Big)+2\alpha_{i+1}\Big(\frac{R+\|a\|_i}{\eta_i}\Big)
  \end{align*}
  because $\alpha_{i+1}$ and $\log(1+\cdots)$ are concave.

  Set then
  $\beta_i(R)=\alpha_i(R+\mu_i)+\log(1+R+\mu_i)+2\log2+2\log3$. We
  will in fact prove the stronger inequality
  $\beta_0(\eta_0\eta_1\cdots\eta_{k-1}\mu_k)\le 2^kB'$ for an
  absolute constant $B'$.  We have
  \begin{align*}
    \beta_i(R) &=\alpha_i(R+\mu_i)+\log(1+R+\mu_i)+2\log2+2\log3\\
    &\le3\log2+2\log3+\log(1+R+\mu_i)+\log\Big(1+\frac{2(R+\|a\|_i)}{\eta_i}\Big)+2\alpha_{i+1}\Big(\frac{R+\mu_i+\|a\|_i}{\eta_i}\Big)\\
    &\le2\left(2\log2+\log3+\log\eta_i+\log\big(1+\frac R{\eta_i}+\mu_{i+1}\big)+\alpha_{i+1}\big(\frac R{\eta_i}+\mu_{i+1}\big)\right)\text{ by Lemma~\ref{lem:etamu}}\\
    &\le2\beta_{i+1}(R/\eta_i).
  \end{align*}

  We therefore have
  \[\alpha_0(\eta_0\cdots\eta_k\mu_k)\le\beta_0(\eta_0\cdots\eta_{k-1}\mu_k)
  \le 2^k\beta_k(\mu_k)\le\alpha_k(2\mu_k)+\log(1+2\mu_k)+\log36,\]
  and $\alpha_k(2\mu_k)$ is bounded by Lemma~\ref{lem:growth0}.
\end{proof}

\subsection{Inverted orbit growth of \boldmath $G_\omega$}\label{ss:inverted}
Recall that, for a group $G$ acting on a set $X$ on the right and a
basepoint $*\in X$, and for a word $w=g_1\dots g_R\in G^*$, we denote
by $\Delta(w)$ the cardinality of the \emph{inverted orbit}
$\{*g_1\cdots g_R,*g_2\cdots g_R,\dots,*g_R,*\}$ of $*$ under $w$. We
denote by $\Delta(R)$ the maximal inverted orbit growth of words $w\in
S^R$. Following the notation used before, we denote by $\Delta_i(R)$
the maximal size of an inverted orbit under a word in $G_i$ of norm
at most $R$ in the metric $\|\cdot\|_i$.

\begin{proposition}\label{prop:invgrowth}
  There is an absolute constant $C$ such that for all $k\in\N$ we have
  \[2^k\le\Delta(\eta_0\cdots\eta_{k-1}\mu_k)\le C2^k.\]
\end{proposition}

\begin{proof}[Proof of the upper bound]
  The calculation in this section follows closely that of the
  previous~\S. Let again $\Delta_i^+$ denote the concave majorand of
  $\Delta_i$. Lemma~\ref{lem:contraction} gives
  \[\Delta_i^+(\eta_i R)\le 2\Delta_{i+1}(R+\|a\|_i/\eta_i).\]
  The same argument as above applies, and we do not repeat it.
\end{proof}

On the other hand, we also obtain a lower bound on the inverted orbit
growth, using a substitution. For that purpose, define
self-substitutions $\zeta_x$ of $\{ab,ac,ad\}$, for
$x\in\{\mathbf0,\mathbf1,\mathbf2\}$, by
\[\begin{matrix}
  \zeta_{\mathbf0}:&ab\mapsto adabac,&ac\to acac,&ad\to adad,\\
  \zeta_{\mathbf1}:&ab\mapsto abab,&ac\to abacad,&ad\to adad,\\
  \zeta_{\mathbf2}:&ab\mapsto abab,&ac\to acac,&ad\to acadab,
\end{matrix}
\]
and note that, for any word $w\in\{ab,ac,ad\}^*$ (representing an
element of $G_{i+1}$) we have
\[\phi_i(\zeta_{\omega_i}(w))\begin{cases}
  \pair{w,w} & \text{ if $\zeta_{\omega_i}(w)$ contains an even number of `$a$'}\\
  \varepsilon\pair{wa,a^{-1}w} & \text{ if $\zeta_{\omega_i}(w)$ contains an odd number of `$a$'}.
\end{cases}
\]
In particular, $\zeta_x$ defines a homomorphism $G_{i+1}\to G_i$.

\begin{proof}[Proof of the lower bound]
  By induction, we see that for any non-trivial $w\in\{ab,ac,ad\}^*$
  (representing an element of $G_k$) we have
  \[\Delta_0(\zeta_{\omega_{0}}\cdots\zeta_{\omega_{k-1}}(w))\ge2^k.\]

  Note then that, if $Z\in\N^3$ count the numbers of $ab,ac,ad$
  respectively in $w$, then $Z^{\mathrm t} M_x$ counts the numbers of
  $ab,ac,ad$ respectively in $\zeta_x(w)$. Let $as\in\{ab,ac,ad\}$ be
  such that $\|s\|_k$ is minimal --- if $\beta\le\gamma,\delta$ then
  $s=b$, etc. Let $W$ be the vector in $\R^3$ with a $1$ at the
  position which $s$ has in $\{ab,ac,ad\}$ and $0$ elsewhere --- if
  $\beta\le\gamma,\delta$ then $W=(1,0,0)^{\mathrm t}$, etc. Set
  $w=\zeta_{\omega_{0}}\cdots\zeta_{\omega_{k-1}}(s)$. We have
  $\Delta_0(w)\ge2^k$, and
  \begin{align*}
    \|w\|_0&= W^{\mathrm t}M_{\omega_{k-1}}\cdots M_{\omega_{0}}V_0\\
    &= \eta_0\cdots\eta_{k-1}W^{\mathrm t}V_k = \eta_0\cdots\eta_{k-1}\mu_k.\qedhere
  \end{align*}
\end{proof}

\subsection{Choice of inverted orbits growth}\label{ss:choice}
If $w$ be a word of norm $R$ over $G$, then its inverted orbit
$\mathcal O_w$ is a subset of $X$ of cardinality at most $R+1$, and
containing $*$. Furthermore, if $w$ is a word over $S$, then $\mathcal
O_w$ is contained in the ball of radius $R$ in the Schreier graph of
$(X,*)$. The set $\{\mathcal O_w\mid w\in S^R\}$ is therefore a finite
subset of the power set of $X$, and we denote its cardinality by
$\Sigma(R)$.

\begin{proposition}\label{prop:invchoices}
  There is an absolute constant $D$ such that for all $k\in\N$ we have
  \[\Sigma(\eta_0\cdots\eta_{k-1}\mu_k)\le D^{2^k}.\]
\end{proposition}
\begin{proof}
  Again we denote by $\Sigma_i(R)$ the choice of inverted orbits
  growth function of $G_i$.  The relevant inequality is
  \[\Sigma_i(\eta_i R)\le\int_{\ell+m=R+\|a\|_i/\eta_i}\Sigma_{i+1}(\ell)\Sigma_{i+1}(m),\]
  and the same argument as in Proposition~\ref{prop:gomega} applies.
\end{proof}
  
\section{Growth of the groups $W_\omega$}
Recall that $W_\omega$ is the permutational wreath product of a group
$A$ with $G_\omega$.  We estimate the growth of $W_\omega$ in terms of
its inverted orbit growth as follows:
\begin{lemma}\label{lem:growthW}
  Let $A$ have growth $v_A$, and let $G$ have growth $v_G$.  Let the
  inverted orbit growth of $G$ on $(X,*)$ be $\Delta$, and let its
  inverted orbit choice growth be $\Sigma$. Assume that $v_A$ is
  log-concave. Consider $W=A\wr_X G$. Then
  \[v_G(R) v_A(R/\Delta(R))^{\Delta(R)}\precsim v_W(R)\precsim v_G(R)
  v_A(R/\Delta(R))^{\Delta(R)}\Sigma(R).\]
\end{lemma}
\begin{proof}
  We begin by the lower bound. For $n\in\N$, consider a word $w$ of
  norm $R$ realizing the maximum $\Delta(R)$; write $\mathcal
  O(w)=\{x_1,\dots,x_k\}$ for $k=\Delta(R)$. Choose then $k$ elements
  $a_1,\dots,a_k$ of norm $\le R/k$ in $A$. Define $f\in
  \sum_XA$ by $f(x_i)=a_i$, all unspecified values being $1$. Then
  $wf\in W$ may be expressed as a word of norm
  $R+|a_1|+\dots+|a_k|\le 2R$ in the standard generators of $W$.

  Furthermore, different choices of $a_i$ yield different elements of
  $W$; and there are $v_A(R/k)^k$ choices for all the elements of
  $A$. This proves the lower bound.

  For the upper bound, consider a word $w$ of norm $R$ in $W$, and
  let $f\in\sum_XA$ denote its value in the base of the wreath
  product. The support of $f$ has cardinality at most $\Delta(R)$, and
  may take at most $\Sigma(R)$ values.

  Write then $\sup(f)=\{x_1,\dots,x_k\}$ for some $k\le\Delta(R)$, and
  let $a_1,\dots,a_k\in A$ be the values of $w$ at its support; write
  $\ell_i=\|a_i\|$. Since $\sum\ell_i\le R$, the norms of the
  different elements on the support of $f$ define a composition of a
  number not greater than $R$ into at most $k$ summands; such a
  composition is determined by $k$ ``marked positions'' among $R+k$,
  so there are at most $\binom{R+k-1}{k}$ possibilities, which we
  bound crudely by $R^k/k!$. Furthermore, if $A$ is finite, then no
  such composition occurs in the count, because the norms of the $a_i$
  are bounded. Each of the $a_i$ is then chosen among $v_A(\ell_i)$.
  elements, and again (by the assumption that $v_A$ is log-concave)
  there are $\prod v_A(\ell_i)\le v_A(R/k)^k$ total choices for the
  elements in $A$. In all cases, therefore, the $\binom{R+k-1}{k}$
  term is absorbed by $v_A(R/k)^k$: if $A$ is finite, then as we
  argued there is no binomial term, while if $A$ is infinite then
  $v_A(R)\succsim R$.

  We have now decomposed $w$ into data that specify it uniquely, and
  we multiply the different possibilities for each of the pieces of
  data. Counting the possibilities for the value of $w$ in $G$, the
  possibilities for its support in $X$, and the possibilities for the
  elements in $A$, we get
  \[v_W(R)\precsim v_G(R)v_A(R/k)^k\Sigma(R),\]
  which is maximized by $k=\Delta(R)$.
\end{proof}

\noindent The previous section then shows:
\begin{corollary}\label{cor:global}
  There are two absolute constants $F,E>1$ such that the growth
  function $v$ of $W_\omega=A\wr_X G_\omega$ satisfies
  \[E^{2^k}\le v(\eta_0\cdots\eta_{k-1}\mu_k)\le F^{2^k}.\]
\end{corollary}
\begin{proof}
  Take together the upper bound on the growth of $G_\omega$ from
  Proposition~\ref{prop:gomega}, the bounds on the inverted orbit
  growth from Proposition~\ref{prop:invgrowth}, and the choices for
  the inverted orbits from Proposition~\ref{prop:invchoices}. The
  conclusion follows from Lemma~\ref{lem:growthW}.
\end{proof}

We now estimate what the growth of $W_\omega$ for periodic sequences
$\omega$. We start by the easy
\begin{lemma}\label{lem:expspaced}
  Let $v,v'\colon\N\to\N$ be increasing functions such that $v(R_t)\le
  v'(R_t)$ for an strictly increasing sequence $R_1,R_2,\dots$ with
  $R_{t+1}/R_t$ bounded. Then $v\precsim v'$.
\end{lemma}
\begin{proof}
  Say $R_{t+1}/R_t\le K$ for all $t\in\N$. Given $R\in\N$, let
  $t\in\N$ be such that $R_{t-1}<R\le R_t$. Then
  \[v(R)\le v(R_t)\le v'(R_t)\le v'(KR),\]
  so $v\precsim v'$.
\end{proof}

The following result isolates a special case of
Theorem~\ref{thm:main}; we include it because it constructs groups
with additional properties (recursive presentation, self-similarity).
\begin{proposition}\label{prop:sprad=>estimate}
  Let $\omega=\overline{\omega_0\dots\omega_{k-1}}$ be a periodic
  sequence. Let $\eta\in(2,3)$ be such that
  \[\eta^k=\operatorname{sp.radius}(M_{\omega_{k-1}}\cdots
  M_{\omega_0}).\]
  Then the group $W_\omega$ has growth
  \[v_\omega(R)\sim\exp(R^{\log2/\log\eta}).\]
\end{proposition}
\begin{proof}
  We choose $V_0\in\Delta$ to be an eigenvector for
  $M_{\omega_{k-1}}\cdots M_{\omega_0}$. Then its eigenvalue is
  $\eta_{k-1}\cdots\eta_0$, and coincides with its spectral
  radius. Indeed the maps $\overline M_x$ are contracting on $\Delta$,
  so $M_{\omega_{k-1}}\cdots M_{\omega_0}$ has precisely one real
  eigenvalue. We get $\eta_{k-1}\cdots\eta_0=\eta^k$. It suffices to
  estimate the growth function $w(R)$ of $W_\omega$ at
  exponentially-spaced values $(\eta_{k-1}\cdots\eta_0)^t$ for
  $t\in\N$, by Lemma~\ref{lem:expspaced}, and at these places we have,
  by Corollary~\ref{cor:global},
  \[E^{2^{kt}}\le v(\eta^{kt}\mu_0)\le F^{2^{kt}}\quad\text{for all }t\in\N,\]
  so, for all $R_t=\eta^{kt}\mu_0$, we get
  \[E^{(R_t/\mu_0)^{\log2/\log\eta}}\le v(R_t)\le
  F^{(R_t/\mu_0)^{\log2/\log\eta}}.\qedhere\]
\end{proof}
Note that, because $(V_0,\omega)$ is periodic, the $V_i$ define a
discrete sequence in $\Delta$ and in particular do not accumulate on
its boundary, so the function $\mu$ is bounded from below on
$\{V_i\}$. More care is needed in the general case.

\subsection{Dynamics on the simplex}
We now show that the spectral radii in
Proposition~\ref{prop:sprad=>estimate} are dense in the interval
$[2,\eta_+]$. For that purpose, it is useful to translate the problem
to a slightly different language.

Let $f$ be the projection of $s^{-1}$ to $\Delta'$. This is a
$3$-to-$1$ map, and is expanding for the Hilbert metric on $\Delta$,
because the $\overline M_x$ are contracting.  Periodic orbits under
$f$ correspond bijectively to $s$-periodic orbits in $\Omega^+$, by
reading them backwards. Indeed, the $\omega$-\coordinate\ can be
uniquely recovered by noting in which subsimplex of $\Delta'$ the
point lies.

More precisely, if $f^kq=q$, then for $i\in\{0,\dots,k\}$ let
$\omega_i$ be such that $f^{k-i}(q)$ belongs to the image of
$\overline M_{\omega_i}$, and extend the sequence $\omega$ periodically. Then
the $f$-orbit of $q$ is the reverse of the $s$-orbit of $(q,\omega)$.

Recall that $\eta$ is a continuous function on the simplex; it equals
$2$ on the boundary, and $3$ at the barycentre of $\Delta$. Write
$\theta(p)=\log\eta(p)$. For a periodic point $p$, of period $k$,
write $\theta^+(p)$ the Cesar\`o average of $\theta$ on $p$:
\[\theta^+(p):=\frac1k\big(\theta(p)+\theta(fp)+\cdots+\theta(f^{k-1}p)\big).\]

\begin{lemma}
  Let $p$ be an $f$-periodic point of period $k$ in $\Delta'$. Let
  $\omega_0,\dots,\omega_{k-1}$ be the corresponding address. Then the
  spectral radius of $M_{\omega_{k-1}}\cdots M_{\omega_0}$ equals
  $k\theta^+(p)$.
\end{lemma}

\begin{proposition}\label{prop:cesaro}
  The averages $\theta^+(p)$ are dense in $[\log2,\log\eta_+]$.
\end{proposition}
\begin{proof}
  Consider $p,p'\in\Sigma$ two periodic points, say of period $k,k'$
  respectively.

  Because $f$ is expanding, there exist arbitrarily small open sets
  $\mathcal U\subset\Delta$ containing $p$, such that $\mathcal
  U\subset f^k(\mathcal U)$ and $f^k$ uniformly expands on $\mathcal
  U$; similarly $f^{k'}$ uniformly expands on the neighbourhood
  $\mathcal U'\subset\Delta$ of $p'$. Without loss of generality, we
  assume $\mathcal U$ and $\mathcal U'$ are relatively compact.

  For any finite sequence
  $\omega_1\dots\omega_n\in\{\mathbf0,\mathbf1,\mathbf2\}^n$, consider
  the fixed point $p_\omega$ of $\overline M_{\omega_n}\cdots\overline
  M_{\omega_1}$. The triangles $\overline M_{\omega_n}\cdots\overline
  M_{\omega_1}(\Delta)$ become arbitrarily small, as $n\to\infty$ and
  all three symbols occur in $\omega_1\dots\omega_n$; so periodic
  points are dense in $\Delta$.

  It then follows that, for every non-empty open set $\mathcal O$, we
  have $\bigcup_{n\ge0}f^n(\mathcal O)=\Delta$; indeed
  $\bigcup_{n\ge0}f^n(\mathcal O)$ is open, and contains all periodic
  points because $f$ is expanding.

  Because $\theta$ is continuous on $\Delta$, for any $\epsilon>0$ we
  can make $\mathcal U$ small enough so that
  $\frac1k(\theta(q)+\theta(fq)+\cdots+\theta(f^{k-1}q))$ is less that
  $\epsilon$ away from $\theta^+(p)$, for all $q\in\mathcal
  U$. Similarly, $\mathcal U'$ may be chosen small enough that the
  average of $\theta$ on the first $k'$ points of the orbit of $q'$ is
  at most $\epsilon$ away from $\theta^+(p')$, for all $q'\in\mathcal
  U'$.

  Since $\mathcal U'$ is relatively compact, there exists $\ell\in\N$
  such that $\mathcal U'\subset f^\ell(\mathcal U)$ and $f^\ell$
  uniformly expands on $\mathcal U$; similarly $\mathcal U\subset
  f^{\ell'}(\mathcal U')$ for some $\ell'\in\N$.

  Consider now $n,n'\in\N$. We have locally defined contractions
  $f^{-k}\colon\mathcal U\to\mathcal U$, $f^{-\ell}\colon\mathcal U'\to\mathcal
  U$, $f^{-k'}\colon\mathcal U'\to\mathcal U'$ and $f^{-\ell'}\colon\mathcal
  U\to\mathcal U'$, which we compose:
  \[\mathcal U\xrightarrow{f^{-\ell'}}\mathcal
  U'\xrightarrow{f^{-n'k'}}\mathcal U'\xrightarrow{f^{-\ell}}\mathcal
  U\xrightarrow{f^{-nk}}\mathcal U.
  \]
  This is a contraction $\mathcal U\to\mathcal U$, so by the Banach
  fixed point theorem there exists a fixed point $q\in\mathcal U$ for
  the composite $f^{\ell+nk+\ell'+n'k'}$ that remains close to the
  orbit of $p$ for $nk$ steps, wanders for $\ell$ steps, remains close
  to the orbit of $p'$ for $n'k'$ steps, and wanders back to $q$ for
  $\ell'$ steps.

  The average of $\theta$ on $q$ is
  \[\theta^+(q)=\frac1{nk+\ell+n'k'+\ell'}\big(\theta(q)+\cdots+\theta(f^{nk+\ell+n'k'+\ell'-1}q)\big).\]
  The orbit of $q$ is, except at $\ell+\ell'$ instants, either close
  to the orbit of $p$ or close to the orbit of $p'$. Consider any
  $\rho\in\mathbb R_+$. Then, as $n,n'\to\infty$ with ratio
  $n/n'\to\rho$, we have
  \[\liminf_{n,n'\to\infty}\theta^+(q)\le\frac{\rho k(\theta^+(p)+\epsilon)+k'(\theta^+(p')+\epsilon)}{\rho k+k'},\]
  and
  \[\limsup_{n,n'\to\infty}\theta^+(q)\ge\frac{\rho k(\theta^+(p)-\epsilon)+k'(\theta^+(p')-\epsilon)}{\rho k+k'}.\]
  Letting $\epsilon$ tend to $0$ and simultaneously considering all
  possible $\rho$, we obtain a dense set of values in
  $[\theta^+(p),\theta^+(p')]$. More precisely, for any
  $t\in[\theta^+(p),\theta^+(p')]$, let $\rho$ be such that $(\rho
  k\theta^+(p)+k'\theta^+(p'))/(\rho k+k')=t$; then for any
  $\epsilon>0$ there exists $\mathcal U,\mathcal U'$ as above; then
  $\ell,\ell'$ as above; and finally $n,n'$ large enough so that
  $|\theta^+(q)-t|<2\epsilon$.

  Finally, note that we can take for $p$ the periodic orbit
  corresponding to the sequence
  $\omega=\overline\zot$; while for $p'$ we
  consider the periodic point corresponding to the sequence
  $\omega'=\overline{\mathbf0^u\mathbf{12}}$ for $u$ sufficiently large; we
  have $\theta^+(p')\to\log2$ as $u\to\infty$.
\end{proof}

\subsection{Proof of Theorem~\ref{thm:periodic}}
Let $\omega$ be a periodic sequence, and let $A$ be a finite group. We
already showed in Proposition~\ref{prop:Wbranched} that $W_\omega$ is
self-similar and branched.  To show that it is contracting, we endow
it with the following metric. It is generated by $\{1,a,b,c,d\}\times
A$; for $s\in\{a,b,c,d\}$ and $t\in A$, the norm of $st$ is $\|s\|$,
while for $t\neq1$ in $A$ its norm is $\|a\|$.

We start by a more general result, which holds for arbitrary sequences
$\omega$. Recall from~\S\ref{ss:metrics} that we set
$G_i=G_{s^i\omega}$ and selected metrics $\|\cdot\|_i$ on $G_i$,
giving constants $\eta_i$. We set $W_i=A\wr_X G_i$, with the metric
above.

The reason we can achieve lower bounds on the growth of $W_i$ may be
illustrated as follows; though we will not use it directly. Consider
the corresponding permutational wreath products $W_i=A\wr_X G_i$,
and note that we also have injective homomorphisms $\psi_i\colon
W_i\to W_{i+1}\wr\sym_2$. The maps $\psi_i$ have the same Lipschitz
property as $\phi_i$, see Lemma~\ref{lem:B}. Additionally, for all
elements of $W_i$ with sufficiently large support (and there are
sufficiently many so as to dominate the asymptotics), we have a
reverse inequality $\|\psi_i(g)\|\ge2/\eta_i\|g\|-C$.

The following combines Lemma~\ref{lem:contraction}
and~\cite{bartholdi-erschler:permutational}*{Lemma~4.2}; we only
sketch the proof since it follows closely that of its models.

\begin{lemma}\label{lem:B}
  For each $i\ge0$ and each $g\in W_i$, with
  $\psi_i(g)=\pair{g_1,g_2}\pi$, we have
  \[\|g_1\|_{i+1}+\|g_2\|_{i+1}\le2\|a\|_{i+1}+\frac2{\eta_i}(\|g\|_i+\|a\|_i).\]
\end{lemma}
\begin{proof}
  Consider a minimal representation $g=s_1t_1\cdots s_nt_n$ of $g\in
  W_i$ with the $s_j$'s in $\{1,a,b,c,d\}$ and the $t_j$'s in
  $A$. Recall that we have $\psi_i:W_i\to W_{i+1}\wr\sym_2$, defined
  on generators by $\psi_i(s)=\phi_i(s)$ for $s\in\{a,b,c,d\}$, and
  $\psi_i(t)=\pair{1,t}$ for $t\in A$. Therefore, $\langle
  b,c,d\rangle$ commutes with $A$ in $W_i$; so (by minimality) we may
  assume that no two consecutive $s_j,s_{j+1}$ belong to $\{b,c,d\}$.

  It follows that we have $g=a^\epsilon x_1t_1ax_2t_2\cdots ax_mt_m$,
  with $\epsilon\in\{0,1\}$, all $x_2,\dots,x_{m-1}\in\{b,c,d\}$,
  $x_1,x_m\in\{1,b,c,d\}$, and $m\le(n+1)/2$.  We then proceed as in
  Lemma~\ref{lem:contraction}, to construct words representing
  $g_1,g_2$. Each $t_j$ contributes a $t_j$ to either $g_1$ or $g_2$;
  each $x_j$ contributes a letter in $\{1,a,b,c,d\}$ to each of $g_1$
  and $g_2$. Therefore, $g_1$ has the form $u_0y_1u_1\cdots y_mu_m$ for
  some $u_j\in A$ and $y_i\in\{1,a,b,c,d\}$, while $g_2$ has the
  similar form $v_0z_1v_1\cdots z_mv_m$. Furthermore,
  $\|y_j\|_{i+1}+\|z_j\|_{i+1}\le2/\eta_i\|ax_j\|$,
  by~\eqref{eq:bcd}. Then
  $\|g_1\|_{i+1}+\|g_2\|_{i+1}\le\|u_0\|_{i+1}+\|v_0\|_{i+1}+2/\eta_i(\|g\|_i+\|a\|_i)$,
  as was to be shown.
\end{proof}

\noindent We finally recall a classical estimate of growth as a
function of sum-contraction:
\begin{lemma}[See e.g.~\cite{bartholdi:upperbd}*{Proposition~4.3}]
  If $G$ is self-similar and sum-contracting with contraction
  $\eta$, then $v_G\precsim\exp(R^{\log d/\log\eta})$.
\end{lemma}

\begin{proof}[Proof of Theorem~\ref{thm:periodic}]
  Lemma~\ref{lem:B} gives an upper bound on the growth of $W_\omega$;
  the same upper bound comes from
  Proposition~\ref{prop:sprad=>estimate}.

  The density of the growth exponents in $[\alpha_-,1]$
  follows from Proposition~\ref{prop:cesaro}.

  To obtain torsion-free examples, apply Lemma~\ref{lem:growthW},
  recalling that the growth of $\tilde G_\omega$ is equivalent to that
  of $\Z\wr G_\omega$, see~\S\ref{ss:torsionfree}.
\end{proof}

\section{Proof of Theorem~\ref{thm:main}}
We now obtain more growth functions by considering $W_\omega$ and
$\tilde G_\omega$ for non-periodic sequences $\omega$. We first give
an estimate for $\mu_k$; only the first part of Lemma~\ref{lem:mu_k}
will be used.

For a finite sequence
$\omega=\omega_0\dots\omega_{n-1}\in\{\mathbf0,\mathbf1,\mathbf2\}^n$
and $p\in\Delta$, we write by extension
\[\overline M_\omega=\overline M_{\omega_0}\cdots\overline M_{\omega_{n-1}}\colon\Delta\to\Delta\]
and
\[\eta(p,\omega_0\dots\omega_{n-1})=\eta(p,\omega_0)\eta(\overline M_{\omega_0}p,\omega_1)\cdots\eta(\overline M_{\omega_1\dots\omega_{n-2}}p,\omega_{n-1}).\]

Fix also once and for all $V_0\in\Delta'$. When a sequence $\omega$ is
under consideration, it defines $V_k\in\Delta'$ and $\eta_k,\mu_k$ by
$V_{k+1}=\overline M_{\omega_k}(V_k)$ and
$\eta_{k+1}=\eta(V_k,\omega_k)$ and $\mu_k=\mu(V_k)$.

\begin{lemma}\label{lem:mu_k}
  Consider $\omega\in\{\mathbf0,\mathbf1,\mathbf2\}^\infty$. For
  $k\in\N$, let $\ell\in\N$ be maximal such that
  $\omega_{k-\ell+1}\dots\omega_k$ contains only two of the symbols
  $\mathbf0,\mathbf1,\mathbf2$, say $\mathbf a,\mathbf b$. Furthermore
  write $\omega_{k-\ell+1}\dots\omega_k=\mathbf a^{i_1}\mathbf
  b^{j_1}\dots\mathbf a^{i_m}\mathbf b^{j_m}$ with $m$ minimal. Then
  \[K/m\le\mu_k\le L/m
  \]
  for absolute constants $K,L$.
\end{lemma}
\begin{proof}
  Without loss of generality, we assume
  $\omega=\dots\mathbf0\mathbf1^t\mathbf2\mathbf0^{i_1}\dots\mathbf1^{j_m}$,
  with $i_1+\dots+j_m=\ell$. We have $M_{\mathbf0^n}=\begin{pmatrix}1
    & 2^n-1 & 2^n-1 \\ 0 & 2^n & 0\\ 0 & 0 & 2^n\end{pmatrix}$.
  
  The image of $\overline M_{\mathbf0}$ is the open triangle spanned
  by $(\frac13,\frac13,\frac13)$, $(\frac12,0,\frac12)$ and
  $(\frac12,\frac12,0)$. The image of $\overline
  M_{\mathbf0\mathbf1^t}$ is contained in the open triangle spanned by
  $(\frac14,\frac12,\frac14)$, $(\frac13,\frac13,\frac13)$ and
  $(\frac12,\frac12,0)$. The image of $\overline
  M_{\mathbf0\mathbf1^t\mathbf2}$ is contained in the open triangle
  spanned by $(\frac15,\frac25,\frac25)$, $(\frac27,\frac27,\frac27)$
  and $(\frac13,\frac13,\frac13)$.  It follows that $V_{k-\ell}$
  belongs to that triangle, so $\mu_{k-\ell}\in(\frac15,\frac13)$
  is bounded away from $0$.

  Now, under application of $\mathbf0,\mathbf1$, the third
  \coordinate\ of $V_{k-\ell+n}$ decreases as $n$ increases, while the
  first two approach $1/2$. To understand how exactly, we consider
  $V=(1/2-\rho\epsilon,1/2-(1-\rho)\epsilon,\epsilon)$ for some
  $\epsilon\ll1$ and $\rho\in(\frac12,1)$, and compute
  \begin{equation}\label{eq:M0^n}
    \overline M_{\mathbf0^n}(V)=
    \begin{pmatrix}\frac12-2^{-n}\rho\epsilon+\mathcal O(\epsilon^2)\\
      \frac12-(1-2^{-n}\rho)\epsilon+\mathcal O(\epsilon^2)\\
      \epsilon-2\rho(1-2^{-n})\epsilon^2+\mathcal
      O(\epsilon^3)\end{pmatrix}.
  \end{equation}
  A similar approximation holds for $\overline M_{\mathbf1^n}(V)$,
  with the first two rows switched.

  Write $W=\overline M_{\mathbf0^n}(V)$. It follows that
  $\mu(W)\approx\epsilon-A\epsilon^2$ is bounded away from $0$ as
  $n\to\infty$, with $A\in(\frac12,1)$. On the other hand, if $n\ge1$,
  then $W$ is of the form
  $(1/2-(1-\rho')\epsilon',1/2-\rho'\epsilon',\epsilon')$ for
  $\epsilon'\approx\epsilon-A\epsilon^2$ and some
  $\rho'\in(\frac12,1)$. Set $X=\overline M_{\mathbf1^p}(W)$; then
  $\mu(X)\approx\epsilon'-2\rho'(1-2^{-p})(\epsilon')^2=\epsilon'-B(\epsilon')^2$
  for some $B\in(\frac12,1)$.

  If we now translate to \coordinate s $\epsilon=1/N$, we get
  $\mu(V)=1/N$, $\mu(W)\approx1/(N+A)$ and $\mu(X)\approx1/(N+A+B)$ with
  $A,B\in(\frac12,1)$. We repeat this $m$ times, giving
  $\mu(V_k)\approx1/(N+A_1+B_1+\dots+A_m+B_m)$ with
  $A_i,B_i\in(\frac12,1)$ if $\mu(V_{k-\ell})\approx1/N$. This
  translates to $K=1,L=2$ in the statement of the lemma. In fact, the
  constants are a bit worse because of the approximations we made
  in~\eqref{eq:M0^n}, that are accurate only for small $\mu$.
\end{proof}

\begin{corollary}\label{cor:mu_k bounded}
  If the sequence $\omega$ has the form
  \begin{equation}\label{eq:omega}
    \omega=(\zot)^{i_1}\mathbf2^{j_1}(\zot)^{i_2}\mathbf2^{j_2}(\zot)^{i_3}\mathbf2^{j_3}\dots,
  \end{equation}
  with $i_1,j_1,i_2,j_2,\dots\ge1$, then the $\mu_k$ are all bounded
  away from $0$.
\end{corollary}
\begin{proof}
  In fact, the image of
  $\overline M_\zot$ is the open triangle spanned by
  $(\frac13,\frac13,\frac13)$, $(\frac27,\frac27,\frac37)$ and
  $(\frac4{17},\frac6{17},\frac7{17})$, so after each $\zot$
  the $\mu_k$ belongs to $(\frac4{17},\frac13)$.

  The image of that triangle under $\overline M_{\mathbf2^i}$ is
  contained in the convex quadrilateral spanned by
  $(\frac13,\frac13,\frac13)$, $(\frac4{17},\frac6{17},\frac7{17})$,
  $(\frac14,\frac14,\frac12)$ and $(\frac15,\frac3{10},\frac12)$, so
  $\mu_k\in(\frac15,\frac13)$ for all $k$.
\end{proof}

We also show that $\eta$ converges very fast to its limiting values
under periodic orbits:
\begin{lemma}\label{lem:expconv}
  There exist constants $A'\le1$, $B'\ge1$ such that
  \begin{enumerate}
  \item\label{lem:expconv1} For all $V\in\Delta'$ and all $n\in\N$,
    \[\eta(V,(\zot)^n)\ge\eta_+^{3n}A';\]
  \item\label{lem:expconv2} For all $V\in\Delta'$ and all $n\in\N$,
    \[\eta(V,\mathbf2^n)\le2^nB'.\]
  \end{enumerate}
\end{lemma}
\begin{proof}
  Let $V_+\in\Delta$ denote the fixed point of $\overline
  M_\zot$. Note first that $\eta(-,\zot)$ is
  differentiable at $V_+$, and that $\overline M_\zot$ is
  uniformly contracting about $V_+$. Let $\mathcal U$ be a
  neighbourhood of $V_+$ such that $\overline M_\zot$ is
  $\rho$-Lipschitz in $\mathcal U$, and let $D$ be an upper bound for
  the derivative of $\log\eta(-,\zot)$ on $\mathcal U$. We may
  assume, without loss of generality, that $\mathcal U$ contains the
  image of $\overline M_\zot$. Recall that
  $\eta_+^3=\eta(V_+,\zot)$.

  For all $k\in\N$, write $V_k=\overline M_{(\zot)^k}(V)$. For
  $k\ge1$ we have $d(V_k,V_+)\le\rho^{k-1}$, so
  $|\log\eta(V_k,\zot)-3\log\eta_+|<D\rho^{k-1}$, while for
  $k=0$ we write
  $|\log\eta(V,\zot)-3\log\eta_+|<3\log3$. Therefore,
  \[|\log\eta(V,(\zot)^n)-3n\log\eta_+|\le
  3\log3+\sum_{k=1}^{n-1}|\log\eta(V_k,\zot)-3\log\eta_+| \le
  3\log3+D/(1-\rho)
  \]
  is bounded over all $n$ and $V$. The estimate~\eqref{lem:expconv1}
  follows, with $A'=3^3\exp(D/(1-\rho))$.

  For the second part, consider $V_k=\overline M_{\mathbf2^k}(V)$, and
  note that $V_k$ converges to a point $V_\infty$ on the side
  $\{\delta=\frac12\}$ of $\Delta$. By the
  approximations~\eqref{eq:M0^n}, we get
  $d(V_k,V_\infty)\approx2^{-k}$, so definitely
  $d(V_k,V_\infty)<\rho^{k-1}$ for some $\rho<1$. As above,
  $\eta(-,\mathbf2)$ is differentiable in a neighbourhood $\mathcal U$
  of $\{\delta=\frac12\}$, and the derivative of
  $\log\eta(-,\mathbf2)$ is bounded on $\mathcal U$, say by $D$. We
  may again assume, without loss of generality, that $\mathcal U$
  contains the image of $\overline M_{\mathbf2}$. Recall that
  $\eta(V,\mathbf2)=2$ for all $V$ on $\{\delta=\frac12\}$.

  As before, for $k\ge1$ we have
  $|\log\eta(V_k,\mathbf2)-\log2|<D\rho^{k-1}$, while for $k=0$ we
  write $|\log\eta(V,\mathbf2)-\log2|<\log3$. Therefore,
  \[|\log\eta(V,\mathbf2^n)-n\log2|\le
  \log3+\sum_{k=1}^{n-1}|\log\eta(V_k,\mathbf2)-\log\eta_+| \le
  \log3+D/(1-\rho)
  \]
  is bounded over all $n$ and $V$. The estimate~\eqref{lem:expconv1}
  follows, with $B'=3\exp(D/(1-\rho))$.
\end{proof}

We now reformulate the statement of Theorem~\ref{thm:main} as follows.
For the main statement, set $g(R)=\log f(R)$; to construct a
torsion-free group as in Remark~\ref{rem:main}, define $g$ implicitly
by $(R/g(R))^{g(R)}=f(R)$. In both cases, we obtain
\begin{equation}\label{eq:g}
  g(2R)\le 2g(R)\le g(\eta_+R)
\end{equation}
for all $R$ large enough. For simplicity (since growth is only an
asymptotic property) we assume that~\eqref{eq:g} holds for all $R$.
Without loss of generality, we also assume that $g$ is increasing and
satisfies $g(1)=1$.

Recall that an initial metric $V_0\in\Delta'$ has been chosen.  We
will construct a sequence $\omega$ such that, for constants $A,B$,
we have
\begin{equation}\label{eq:AB}
  A\le \frac{g(\eta(V_0,\omega_0\dots\omega_{k-1}))}{2^k}\le B\text{ for
  all }k;
\end{equation}
in fact, it will suffice to obtain this inequality for a set of values
$k_0,k_1,\dots$ of $k$ such that $\sup_i(k_{i+1}-k_i)<\infty$. Indeed,
the orbit $V_i$ of $V_0$ in $\Delta'$ will remain bounded, so we will
have $\mu(V_i)\in[C,1]$ for some $C>0$. Consider for simplicity the
first part of the theorem, for which $g(R)=\log f(R)$. Then, by
Corollary~\ref{cor:global},
\begin{align*}
  E^{2^k}\le v(\eta(V_0,\omega_0\dots\omega_{k-1})\mu_k)\le &v(Bg^{-1}(2^k)),\\
  &v(ACg^{-1}(2^k))\le v(\eta(V_0,\omega_0\dots\omega_{k-1})\mu_k)\le F^{2^k}
\end{align*}
and therefore $v(R)\sim \exp(g(R))=f(R)$. Similar considerations hold
for the torsion-free case of Remark~\ref{rem:main}, using
Lemma~\ref{lem:growthW}.

\begin{proof}[Proof of~\eqref{eq:AB}]
  We will construct a sequence $\omega$ of the form~\eqref{eq:omega},
  \[\omega=(\zot)^{i_1}\mathbf2^{j_1}(\zot)^{i_2}\mathbf2^{j_2}(\zot)^{i_3}\mathbf2^{j_3}\dots,
  \]
  with $i_1,j_1,i_2,j_2,\dots\ge1$. The $\mu_k$ are bounded away from
  $0$ by Corollary~\ref{cor:mu_k bounded}.

  We start by the empty sequence. Then, assuming
  $\omega'=\omega_0\dots\omega_{k-1}$ has been constructed, we repeat the
  following:
  \begin{itemize}
  \item while $g(\eta(V_0,\omega'))<2^k$, we append $\zot$ to
    $\omega'$;
  \item while $g(\eta(V_0,\omega'))>2^k$, we append $\mathbf2$ to $\omega'$.
  \end{itemize}

  For our induction hypothesis, we assume that the stronger condition
  \[\frac122^k\le g(\eta(V_0,\omega'))\le 2^k\]
  holds for each $k$ of the form $i_1+j_1+\dots+i_m+j_m$, and that
  \[2^k\le g(\eta(V_0,\omega'))\le3^32^k\] holds for each $k$ of the form
  $i_1+j_1+\dots+i_m$; these conditions apply whenever $\omega$ is a
  product of `syllables' $(\zot)^{i_t}$ and $\mathbf2^{j_t}$.

  Consider first the case $\frac122^k\le g(\eta(V_0,\omega'))\le2^k$;
  and let $n$ be minimal such that
  $g(\eta(V_0,\omega'(\zot)^n))>2^{k+3n}$. Then, for all
  $i\in\{1,\dots,n\}$, Lemma~\ref{lem:expconv}\eqref{lem:expconv1}
  gives
  $\eta(V_0,\omega'(\zot)^i)\ge\eta(V_0,\omega')\eta_+^{3i}A'$.
  Let $u\in\N$ be minimal such that $A'\ge\eta_+^{-u}$; this, combined
  with $g(\eta_+R)\ge2g(R)$, gives
  \[g(\eta(V_0,\omega'(\zot)^i))\ge g(\eta(V_0,\omega')\eta_+^{3i-u})\ge
  2^{-1-u}2^{k+3i}.
  \]
  By minimality of $n$, we have
  $g(\eta(V_0,\omega'(\zot)^{n-1}))\le2^{k+3(n-1)}$; since $g$
  is sublinear and $\eta\le3$, we get
  \[g(\eta(V_0,\omega'(\zot)^n))\le3^32^{k+3n}.\]

  Consider then the case $2^k\le g(\eta(V_0,\omega'))\le3^32^k$, which
  is similar, and let $n$ be minimal such that
  $g(\eta(V_0,\omega'\mathbf2^n))<2^{k+n}$. Then, for all
  $i\in\{1,\dots,n\}$, Lemma~\ref{lem:expconv}\eqref{lem:expconv2}
  gives $\eta(V_0,\omega'\mathbf2^i)\le\eta(V_0,\omega')2^iB'$; this,
  combined with $g(2R)\le2g(R)$, gives
  \[g(\eta(V_0,\omega'\mathbf2^i))\le3^32^{k+i}B'.
  \]
  By minimality of $n$, we have
  $g(\eta(V_0,\omega'\mathbf2^{n-1}))\ge2^{k+n-1}$; since $g$ is
  increasing, we get
  \[g(\eta(V_0,\omega'\mathbf2^n))\ge\frac122^{k+n}.
  \]
  We have proved the claim~\eqref{eq:AB}, with
  $A=2^{-1-u}$ and $B=3^3B'$.
\end{proof}

\begin{remark}\label{rem:solvablewp}
  The construction of $\omega$ from $f$ is algorithmic, in the
  following sense. If $V_0$ is chosen with rational \coefficient s,
  then all $V_k$ have rational \coefficient s, and therefore are be
  computable. Furthermore, $\eta(V_0,\omega_0\dots\omega_{k-1})$ is
  also computable. Therefore, if $f$ is recursive, then so is
  $\omega$.

  It then follows that $G_\omega$, $\tilde G_\omega$ and
  $W_\omega=A\wr_X G_\omega$ are recursively presented (for
  recursively presented $A$), see~\S\ref{ss:womega}
  and~\S\ref{ss:presentations}.
\end{remark}

\subsection{Illustrations}\label{ss:examples}
Given a sufficiently regular growth function $f$,
Theorem~\ref{thm:main} constructs a sequence $\omega$ such that the
growth of $W_\omega$ is asymptotically $f$. We may proceed the other
way round, and consider `regular' sequences $\omega$, using then
Corollary~\ref{cor:global}, Corollary~\ref{cor:mu_k bounded} and
Lemma~\ref{lem:expconv} to estimate the growth of $W_\omega$. Here are
four examples:
\begin{itemize}
\item Consider the sequence
  $\omega=(\zot)\mathbf2^1(\zot)\mathbf2^2(\zot)\mathbf2^3(\zot)\mathbf2^4\dots$. Among the
  first $k$ entries, approximately $\sqrt k$ instances of $\zot$ will
  have been seen; therefore
  $\eta(V_0,\omega_0\dots\omega_{k-1})\approx 2^{k+\mathcal O(1)\sqrt
    k}$. This gives a growth function of the order of
  \[\exp\big(R/\exp(\mathcal O(1)\sqrt{\log R})\big).\]

\item Consider the sequence
  $\omega=(\zot)\mathbf2^1(\zot)\mathbf2^2(\zot)\mathbf2^4(\zot)\mathbf2^8\dots$. Among the
  first $k$ entries, approximately $\log k$ instances of $\zot$ will
  have been seen; therefore
  $\eta(V_0,\omega_0\dots\omega_{k-1})\approx 2^{k+\mathcal O(1)\log
    k}$. This gives a growth function of the order of
  \[\exp\big(R/(\log R)^{\mathcal O(1)}\big).\]

\item Consider the sequence
  $\omega=(\zot)\mathbf2^{2^1}(\zot)\mathbf2^{2^2}(\zot)\mathbf2^{2^4}(\zot)\mathbf2^{2^8}\dots$.
  Among the first $k$ entries, approximately $\log\log k$ instances of
  $\zot$ will have been seen; therefore
  $\eta(V_0,\omega_0\dots\omega_{k-1})\approx 2^{k+\mathcal
    O(1)\log\log k}$. This gives a growth function of the rough order
  of
  \[\exp\big(R/(\log\log R)^{\mathcal O(1)}\big).\]

\item Consider the Ackermann function
  \[A(m,n)=\begin{cases}
    n+1 & \text{ if }m=0,\\
    A(m-1,1) & \text{ if $m>0$ and }n=0,\\
    A(m-1,A(m,n-1)) & \text{ if $m>0$ and }n>0,
  \end{cases}
  \]
  and consider
  $\omega=(\zot)\mathbf2^{A(0,0)}(\zot)\mathbf2^{A(1,1)}(\zot)\mathbf2^{A(2,2)}\dots$.
  Then $W_\omega$ is a group of subexponential growth, whose growth is
  larger than any primitive recursive function of subexponential
  growth.
\end{itemize}

\subsection*{Thanks}
We express our gratitude to Bill Thurston for explanations and
discussions on the topic of dynamics on the simplex and on the density
of $\eta$-averages as in Theorem~\ref{thm:periodic}, which he
generously contributed during a conference in Roskilde in September
2010.

We are also grateful to Slava Grigorchuk, Pierre de la Harpe, Martin
Kassabov and Igor Pak for their comments on a previous version of this
article, that helped improve its readability.



\bib{bartholdi:upperbd}{article}{
  author={Bartholdi, Laurent},
  title={The growth of Grigorchuk's torsion group},
  journal={Internat. Math. Res. Notices},
  date={1998},
  number={20},
  pages={1049--1054},
  issn={1073-7928},
  review={\MR {1656258 (99i:20049)}},
  doi={10.1155/S1073792898000622},
  eprint={arXiv:math/0012108},
}

\bib{bartholdi:lowerbd}{article}{
  author={Bartholdi, Laurent},
  title={Lower bounds on the growth of a group acting on the binary rooted tree},
  journal={Internat. J. Algebra Comput.},
  volume={11},
  date={2001},
  number={1},
  pages={73--88},
  issn={0218-1967},
  review={\MR {1818662 (2001m:20044)}},
  doi={10.1142/S0218196701000395},
  eprint={arXiv:math/9910068},
}

\bib{bartholdi-s:wpg}{article}{
  author={Bartholdi, Laurent},
  author={{\v {S}}uni{\'k}, Zoran},
  title={On the word and period growth of some groups of tree automorphisms},
  journal={Comm. Algebra},
  volume={29},
  date={2001},
  number={11},
  pages={4923--4964},
  issn={0092-7872},
  review={\MR {1856923 (2002i:20040)}},
  doi={10.1081/AGB-100106794},
  eprint={arXiv:math/0005113},
}

\bib{bartholdi:lpres}{article}{
  author={Bartholdi, Laurent},
  title={Endomorphic presentations of branch groups},
  journal={J. Algebra},
  volume={268},
  date={2003},
  number={2},
  pages={419--443},
  issn={0021-8693},
  review={\MR {2009317 (2004h:20044)}},
  doi={10.1016/S0021-8693(03)00268-0},
  eprint={arXiv:math/0007062},
}

\bib{bartholdi-g-s:bg}{article}{
  author={Bartholdi, Laurent},
  author={Grigorchuk, Rostislav I.},
  author={{\v {S}}uni{\'k}, Zoran},
  title={Branch groups},
  conference={ title={Handbook of algebra, Vol. 3}, },
  book={ publisher={North-Holland}, place={Amsterdam}, },
  date={2003},
  pages={989--1112},
  review={\MR {2035113 (2005f:20046)}},
  doi={10.1016/S1570-7954(03)80078-5},
  eprint={arXiv:math/0510294},
}

\bib{bartholdi-erschler:permutational}{article}{
  author={Bartholdi, Laurent},
  author={Erschler, Anna G.},
  title={Growth of permutational extensions},
  journal={Invent. Math.},
  volume={189},
  date={2012},
  number={2},
  pages={431--455},
  issn={0020-9910},
  review={\MR {2947548}},
  doi={10.1007/s00222-011-0368-x},
  eprint={arXiv:math/1011.5266},
}

\bib{bartholdi-smoktunowicz:images}{article}{
  author={Bartholdi, Laurent},
  author={Smoktunowicz, Agata},
  title={Images of Golod-Shafarevich algebras with small growth},
  date={2013},
  status={to appear in Quartely J. Math},
  eprint={arXiv:math/1108.4267},
}

\bib{bartholdi-erschler:boundarygrowth}{article}{
  author={Bartholdi, Laurent},
  author={Erschler, Anna G.},
  title={Poisson-Furstenberg boundary and growth of groups},
  date={2011},
  status={submitted},
  eprint={arXiv:math/1107.5499},
}

\bib{birkhoff:jentzsch}{article}{
  author={Birkhoff, Garrett},
  title={Extensions of Jentzsch's theorem},
  journal={Trans. Amer. Math. Soc.},
  volume={85},
  date={1957},
  pages={219--227},
  issn={0002-9947},
  review={\MR {0087058 (19,296a)}},
}

\bib{brieussel:growth}{article}{
  author={Brieussel, J\'er\'emie},
  title={Growth behaviours in the range $e^{(r^\alpha )}$},
  date={2011},
  eprint={arXiv:math/1107.1632},
}

\bib{brieussel:phd}{thesis}{
  author={Brieussel, J\'er\'emie},
  title={Growth of certain groups of automorphisms of rooted trees},
  language={French},
  type={Doctoral Dissertation},
  place={Universit\'e de Paris 7},
  year={2008},
  url={http://www.institut.math.jussieu.fr/\char 126theses~/2008/brieussel/},
}

\bib{cornulier:fpwreath}{article}{
  author={de Cornulier, Yves},
  title={Finitely presented wreath products and double coset decompositions},
  journal={Geom. Dedicata},
  volume={122},
  date={2006},
  pages={89--108},
  issn={0046-5755},
  review={\MR {2295543 (2008e:20040)}},
  doi={10.1007/s10711-006-9061-4},
}

\bib{erschler:boundarysubexp}{article}{
  author={Erschler, Anna G.},
  title={Boundary behavior for groups of subexponential growth},
  journal={Ann. of Math. (2)},
  volume={160},
  date={2004},
  number={3},
  pages={1183\ndash 1210},
  issn={0003-486X},
  review={MR2144977},
}

\bib{erschler:degrees}{article}{
  author={Erschler, Anna G.},
  title={On the degrees of growth of finitely generated groups},
  language={Russian},
  journal={Funktsional. Anal. i Prilozhen.},
  volume={39},
  date={2005},
  number={4},
  pages={86--89},
  issn={0374-1990},
  translation={ journal={Funct. Anal. Appl.}, volume={39}, date={2005}, number={4}, pages={317--320}, issn={0016-2663}, },
  review={\MR {2197519 (2006k:20056)}},
  doi={10.1007/s10688-005-0055-z},
}

\bib{erschler:critical}{article}{
  author={Erschler, Anna G.},
  title={Critical constants for recurrence of random walks on $G$-spaces},
  journal={Ann. Inst. Fourier (Grenoble)},
  volume={55},
  date={2005},
  number={2},
  pages={493--509},
  issn={0373-0956},
  review={\MR {2147898 (2006c:20085)}},
}

\bib{grigorchuk:growth}{article}{
  author={Grigorchuk, Rostislav~I.},
  title={On the Milnor problem of group growth},
  date={1983},
  issn={0002-3264},
  journal={Dokl. Akad. Nauk SSSR},
  volume={271},
  number={1},
  pages={30\ndash 33},
  review={\MRhref {85g:20042}},
}

\bib{grigorchuk:gdegree}{article}{
  author={Grigorchuk, Rostislav~I.},
  title={Degrees of growth of finitely generated groups and the theory of invariant means},
  date={1984},
  issn={0373-2436},
  journal={{\cyreight Izv. Akad. Nauk SSSR Ser. Mat.}},
  volume={48},
  number={5},
  pages={939\ndash 985},
  note={English translation: {Math. USSR-Izv. \textbf {25} (1985), no.~2, 259--300}},
  review={\MRhref {86h:20041}},
}

\bib{grigorchuk:pgps}{article}{
  author={Grigorchuk, Rostislav~I.},
  title={Degrees of growth of $p$-groups and torsion-free groups},
  date={1985},
  issn={0368-8666},
  journal={Mat. Sb. (N.S.)},
  volume={126(168)},
  number={2},
  pages={194\ndash 214, 286},
  review={\MRhref {86m:20046}},
}

\bib{grigorchuk-m:interm}{article}{
  author={Grigorchuk, Rostislav~I.},
  author={Mach{\`\i }, Antonio},
  title={An example of an indexed language of intermediate growth},
  date={1999},
  issn={0304-3975},
  journal={Theoret. Comput. Sci.},
  volume={215},
  number={1-2},
  pages={325\ndash 327},
  review={\MRhref {99k:68092}},
}

\bib{gromov:nilpotent}{article}{
  author={Gromov, Mikhael L.},
  title={Groups of polynomial growth and expanding maps},
  date={1981},
  issn={0073-8301},
  journal={Inst. Hautes {\'E}tudes Sci. Publ. Math.},
  number={53},
  pages={53\ndash 73},
}

\bib{harpe:ggt}{book}{
  author={Harpe, Pierre~{de la}},
  title={Topics in geometric group theory},
  publisher={University of Chicago Press},
  address={Chicago, IL},
  date={2000},
  isbn={0-226-31719-6; 0-226-31721-8},
  review={\MRhref {2001i:20081}},
}

\bib{kassabov-pak:growth}{article}{
  title={Groups of oscillating intermediate growth},
  author={Martin Kassabov},
  author={Igor Pak},
  eprint={arXiv:math/1108.0262},
  date={2011},
}

\bib{krause-l:gkdim}{book}{
  author={Krause, G{\"u}nter R.},
  author={Lenagan, Thomas H.},
  title={Growth of algebras and Gelfand-Kirillov dimension},
  series={Graduate Studies in Mathematics},
  volume={22},
  edition={Revised edition},
  publisher={American Mathematical Society},
  place={Providence, RI},
  date={2000},
  pages={x+212},
  isbn={0-8218-0859-1},
  review={\MRhref {2000j:16035}},
}

\bib{leonov:lowerbd}{article}{
  author={Leonov, Yuri{\u \i }~G.},
  title={On a lower bound for the growth function of the Grigorchuk group},
  date={2000},
  issn={0025-567X},
  journal={Mat. Zametki},
  volume={67},
  number={3},
  pages={475\ndash 477},
  review={\MRhref {1 779 480}},
}

\bib{mann:howgroupsgrow}{incollection}{
  author={Mann, Avinoam},
  title={How groups grow},
  date={2012},
  series={London Math. Soc. Lecture Note Ser.},
  volume={395},
  publisher={Cambridge Univ. Press},
  address={Cambridge},
  pages={1\ndash 200},
}

\bib{milnor:5603}{article}{
  author={Milnor, John~W.},
  title={Problem 5603},
  date={1968},
  journal={Amer. Math. Monthly},
  volume={75},
  pages={685\ndash 686},
}

\bib{muchnik-p:growth}{article}{
  author={Muchnik, Roman},
  author={Pak, Igor},
  title={On growth of Grigorchuk groups},
  date={2001},
  issn={0218-1967},
  journal={Internat. J. Algebra Comput.},
  volume={11},
  number={1},
  pages={1\ndash 17},
  review={\MRhref {1 818 659}},
}

\bib{sidki:presentation}{article}{
  author={Sidki, Said},
  title={On a $2$-generated infinite $3$-group: the presentation problem},
  journal={J. Algebra},
  volume={110},
  date={1987},
  number={1},
  pages={13--23},
  issn={0021-8693},
  review={\MR {904179 (89b:20081a)}},
  doi={10.1016/0021-8693(87)90034-2},
}

\bib{trofimov:growth}{article}{
  author={Trofimov, V. I.},
  title={The growth functions of finitely generated semigroups},
  journal={Semigroup Forum},
  volume={21},
  date={1980},
  number={4},
  pages={351--360},
  issn={0037-1912},
  review={\MR {597500 (82g:20094)}},
  doi={10.1007/BF02572559},
}

\bib{warfield:tensor}{article}{
  author={Warfield, Robert B., Jr.},
  title={The Gel\cprime fand-Kirillov dimension of a tensor product},
  journal={Math. Z.},
  volume={185},
  date={1984},
  number={4},
  pages={441--447},
  issn={0025-5874},
  review={\MR {733766 (85f:17006)}},
}



\begin{bibdiv}
\begin{biblist}
\font\cyreight=wncyr8
\bibselect{math}
\end{biblist}
\end{bibdiv}
\end{document}